\documentclass[letterpaper, 10 pt, conference]{ieeeconf}  % Comment this line out
\IEEEoverridecommandlockouts                              % This command is only
\usepackage{url}  %Required
\usepackage{graphicx}  %Required
\usepackage{amsmath,amssymb,amsfonts,mathrsfs,bm} 
\usepackage{caption}
\usepackage{cleveref}
\usepackage{array}
\usepackage{scalerel}
\usepackage{algorithm}
\usepackage{booktabs}
\usepackage{epstopdf}
\usepackage{multirow}
\usepackage{rotating}
\usepackage{algpseudocode}
\usepackage{cite}
\usepackage{epsfig}
\usepackage{tabularx}
\usepackage[table]{xcolor}
\usepackage{epstopdf}
\usepackage{accents}
\usepackage{circuitikz}
\usepackage{textcomp}
\usepackage{pgfplots}
\usepackage{lscape}
\pgfplotsset{compat=1.14}

\usepackage{lipsum}
\usepackage{mathtools}
\usetikzlibrary{arrows,shapes,positioning}
\usepackage{xspace}
\newcommand{\matpower}[0]{\textsc{Matpower}\xspace}

\ifx\theorem\undefined
\newtheorem{theorem}{Theorem}
\newenvironment{theorem*}{\par\noindent{\bf Theorem\ }}{\hfill\\[2mm]}

\newenvironment{corollary*}{\par\noindent{\bf Corollary\ }}{\hfill\\[2mm]}
\newtheorem{definition}{Definition}

\fi

\newcommand{\tr}{\mathrm{tr}}
\newcommand{\onebf}{\mathbf{1}}
\newcommand{\msh}{\!\!\:}

\newcommand{\Abf}{\mathbf{A}}

\newcommand{\bbf}{\mathbf{b}}

\newcommand{\Cbf}{\mathbf{C}}
\newcommand{\cbf}{\mathbf{c}}
\newcommand{\Ccal}{\mathcal{C}}
\newcommand{\Cbb}{\mathbb{C}}

\newcommand{\dbf}{\mathbf{d}}

\newcommand{\ebf}{\mathbf{e}}

\newcommand{\Ecal}{\mathcal{E}}

\newcommand{\fbf}{\mathbf{f}}

\newcommand{\Fcal}{\mathcal{F}}

\newcommand{\Gcal}{\mathcal{G}}
\newcommand{\gbf}{\mathbf{g}}

\newcommand{\Hbf}{\mathbf{H}}
\newcommand{\Hcal}{\mathcal{H}}
\newcommand{\Hbb}{\mathbb{H}}

\newcommand{\Ibf}{\mathbf{I}}
\newcommand{\irm}{\mathrm{i}}

\newcommand{\Jbf}{\mathbf{J}}

\newcommand{\Mbf}{\mathbf{M}}

\newcommand{\obf}{\mathbf{o}}

\newcommand{\pbf}{\mathbf{p}}

\newcommand{\qbf}{\mathbf{q}}

\newcommand{\Rbb}{\mathbb{R}}

\newcommand{\sbf}{\mathbf{s}}

\newcommand{\Ubf}{\mathbf{U}}
\newcommand{\rbf}{\mathbf{r}}
\newcommand{\vbf}{\mathbf{v}}

\newcommand{\Vcal}{\mathcal{V}}

\newcommand{\Wbf}{\mathbf{W}}

\newcommand{\xbf}{\mathbf{x}}

\newcommand{\ybf}{\mathbf{y}}
\newcommand{\Ybf}{\mathbf{Y}}

\newcommand{\diagrm}{\mathrm{diag}}
\newcommand{\realrm}{\mathrm{real}}

\newcommand{\shrm}{\mathrm{sh}}
\newcommand{\maxrm}{\textnormal{max}}
\newcommand{\minrm}{\textnormal{min}}

\newcommand{\sm}[2]{\scaleto{#1\mathstrut}{#2pt}}
\definecolor{Fcolor}{rgb}{0, 0.5, 0.25}
\newif\ifcomment
\newcommand{\smallMinus}{\scalebox{1.2}[0.7]{-}}

\crefrangelabelformat{equation}{(#3#1#4)\,--\,(#5#2#6)}
\crefmultiformat{equation}{(#2#1#3)}{\,--\,(#2#1#3)}{#2#1#3}{#2#1#3}
\crefname{equation}{}{}

\crefmultiformat{table}{(#2#1#3)}{\,--\,(#2#1#3)}{#2#1#3}{#2#1#3}
\crefname{table}{}{}%
\crefname{figure}{Figure}{Figures}
\crefname{algorithm}{Algorithm}{}
\crefname{table}{Table}{Tables}
\crefname{lemma}{Lemma}{Lemmas}
\crefname{theorem}{Theorem}{Theorems}
\crefname{section}{Section}{Sections}
\crefname{definition}{Definition}{Definitions}
\commentfalse

\makeatletter
\DeclareRobustCommand{\cev}[1]{%
  \mathpalette\do@cev{#1}%
}
\newcommand{\do@cev}[2]{%
  \fix@cev{#1}{+}%
  \reflectbox{$\m@th#1\vec{\reflectbox{$\fix@cev{#1}{-}\m@th#1#2\fix@cev{#1}{+}$}}$}%
  \fix@cev{#1}{-}%
}
\newcommand{\fix@cev}[2]{%
  \ifx#1\displaystyle
    \mkern#2 1mu
  \else
    \ifx#1\textstyle
      \mkern#2 3mu
    \else
      \ifx#1\scriptstyle
        \mkern#2 2mu
      \else
        \mkern#2 2mu
      \fi
    \fi
  \fi
}

\title{\LARGE \bf
Penalized Parabolic Relaxation for Optimal Power Flow Problem
}
\author{Fariba Zohrizadeh, Mohsen Kheirandishfard, Edward Quarm Jnr., and Ramtin Madani
\thanks{Fariba Zohrizadeh and Mohsen Kheirandishfard are with the Department of Computer Science and Engineering, University of Texas at Arlington, (email: fariba.zohrizadeh@uta.edu, mohsen.kheirandishfard@uta.edu). Edward Quarm and Ramtin Madani are with the Department of Electrical Engineering, University of Texas at Arlington (email: edwardarthur.quarmjnr@uta.edu, ramtin.madani@uta.edu). This work is in part supported by the NSF award 1809454 and a University of Texas System STARs award.
}%
}

\begin{document}

\maketitle
\thispagestyle{empty}
\pagestyle{empty}
%%%%%%%%%%%%%%%%%%%%%%%%%%%%%%%%%%%%%%%%%%%%%%%%%%%%%%%%%%%%%%%%%%%%%%%%%%%%%%%%
\begin{abstract}
This paper is concerned with optimal power flow (OPF), which is the problem of optimizing the transmission of electricity in power systems. Our main contributions are as follows: (i) we propose a novel {\it parabolic relaxation}, which transforms non-convex OPF problems into convex quadratically-constrained quadratic programs (QCQPs) and can serve as an alternative to the common practice semidefinite programming (SDP) and second-order cone programming (SOCP) relaxations, (ii) we propose a penalization technique which is compatible with the SDP, SOCP, and parabolic relaxations and guarantees the recovery of feasible solutions for OPF, under certain assumptions. The proposed penalized convex relaxation can be used sequentially to find feasible and near-globally optimal solutions for challenging instances of OPF. Extensive numerical experiments on small and large-scale benchmark systems corroborate the efficacy of the proposed approach. By solving a few rounds of penalized convex relaxation, fully feasible solutions are obtained for benchmark test cases from \cite{coffrin2014nesta,bukhsh2013local,zimmerman2011matpower} with as many as 13659 buses. In all cases, the solutions obtained are not more than 0.32\% worse than the best-known solutions.
\end{abstract}
%%%%%%%%%%%%%%%%%%%%%%%%%%%%%%%%%%%%%%%%%%%%%%%%%%%%%%%%%%%%%%%%%%%%%%%%%%%%%%%%
\section{Introduction}
The optimal power flow problem (OPF) is concerned with the optimization of voltages, power flows, and power injections across transmission and distribution networks.
%, for the optimal transmission and distribution of electricity. 
This problem can be formulated as the minimization of a cost function (e.g., generation cost) subject to nonlinear constraints on power and voltage variables. Due to the inherent complexity of physical laws that model the flow of electricity, some of these constraints are non-convex, which makes the OPF problem NP-hard in general \cite{lehmann2016ac,verma2010power}.
Substantial research efforts have been devoted to this fundamental problem since the 1960s \cite{carpentier1962contribution}. %Many studies endeavored to tackle the problem using 
Conventional methods for solving OPF include, linear approximations, local search algorithms, particle swarm optimization, fuzzy logic (see \cite{pandya2008survey,momoh1999review1,momoh1999review2} and the references therein). However, the existing methods do not offer guaranteed recovery of globally optimal solutions or even feasible points \cite{castillo2013computational}.

One of the most promising approaches to OPF is semidefinite programming (SDP) relaxation, %\cite{bai2008semidefinite,lavaei2012zero} 
which is proven to be exact for a variety of benchmark instances \cite{lavaei2012zero}. %and in general, it can offer a lower bound on the unknown globally optimal cost of any OPF problem \cite{lavaei2012zero}.
%Moreover, for cases where the relaxation is not exact, 
In general, the solution of SDP relaxation offers a lower bound for the unknown globally optimal cost of OPF. 
In order to address the inexactness of SDP relaxation for challenging instances of OPF (e.g., \cite{bukhsh2013local,lesieutre2011examining,molzahn2016convex}), further investigation and improvement are carried out in \cite{zhang2013geometry,kocuk2016inexactness,coffrin2017strengthening,chen2016bound,louca2014nondegeneracy,madani2015convex,gan2015exact,bose2015quadratically}. Since inexact convex relaxations may not lead to physically meaningful solutions for OPF, alternative strategies are proposed to infer OPF feasible points from inexact convex relaxations. For instance, branch-and-bound algorithms \cite{gopalakrishnan2012global,phan2012lagrangian} iteratively partition search spaces to find tighter relaxations. In \cite{madani2015convex,madani2016promises,natarajan2013penalized}, penalty terms are incorporated into the objective of convex relaxations in order ensure OPF feasibility. Moment relaxation algorithms \cite{molzahn2015solution,josz2015application,molzahn2015sparsity} form hierarchies of SDP relaxations %\cite{lasserre2009moments} 
to obtain globally optimal solutions for OPF. Most recently, \cite{wei2017optimal} proposes a sequential convex optimization method with the aim of recovering OPF feasible points.

In addition to the exactness issues, SDP relaxation suffers from high computational cost due to the presence of high-order conic constraints. This shortcoming limits the applicability of SDP relaxation especially for large-scale instances of the OPF problem. To overcome this issue and enhance the scalability of SDP relaxation, some studies propose computationally-cheaper relaxations including second-order cone programming (SOCP) \cite{kocuk2016strong,madani2017scalable}, quadratic programming (QP) \cite{coffrin2016qc,marley2017solving}, linear programming (LP) \cite{bienstock2014linear,coffrin2014linear}. Some papers have leveraged the sparsity of power networks to decompose large-scale conic constraints into lower order ones \cite{molzahn2013implementation,andersen2014reduced,bose2015equivalent,madani2016promises,guo2017case}. 
Additionally, several extensions of OPF have been recently studied under more general settings, to address considerations such as the security of operation \cite{madani2016promises,capitanescu2011state,dvorkin2016optimizing}, robustness \cite{dorfler2016breaking}, energy storage \cite{marley2017solving}, distributed platforms \cite{lam2012distributed,dall2013distributed}, and uncertainty of generation \cite{anese2017chance}. %(see \cite{capitanescu2016critical} and the references therein). %The main contributions of the present work are outlined in the following subsection.

%=============Contributions====================================
%\subsection{Contributions}
In this paper, we introduce a novel and computationally-efficient parabolic relaxation and investigate its relation with the common practice SDP and SOCP relaxations. The proposed parabolic relaxation relies on convex quadratic inequalities only, as opposed to conic constraints. A penalization method is introduced for finding feasible and near-globally optimal solutions, which is compatible with the SDP, SOCP, and parabolic relaxations. We offer theoretical guarantees for the recovery of feasible solutions for OPF using penalization.
% Notation===============================% unless otherwise mentioned
\subsection{Notations}
Throughout this paper, matrices, vectors, and scalars are represented by bold uppercase, bold lowercase, and italic lowercase letters, respectively. The symbols $\Rbb$, $\Cbb$, and $\Hbb_n$ denote the sets of real numbers, complex numbers, and $n\times n$ Hermitian matrices, respectively. The notation ``$\irm$'' is reserved for the imaginary unit. $|\cdot|$ represents the absolute value of a scalar or the cardinality of a set, depending on the context. The symbols $\!(\cdot)^\ast\!$ and $\!(\cdot)^{\!\top}\!$ represent the conjugate transpose and transpose operators, respectively. %For every positive integers $m$ and $n$, 
The notations $\mathbf{I}_n$ and $\mathbf{0}_{m\times n}$ refer to the $n\times n$ identity and $m\times n$ zero matrices, respectively. Given an $n\times 1$ vector $\mathbf{x}$, the notation $[\mathbf{x}]$ refers to the $n\times n$ diagonal matrix with the elements of $\mathbf{x}$ on the diagonal. The symbols $\lambda_{\min}(.)$ and $\lambda_{\max}(.)$ denote the minimum and maximum eigenvalues, respectively. Given a matrix $\Abf$, the notation $A_{jk}$ refers to its $(j,k)$ entry. $\Abf\succeq 0$ means that $\Abf$ is symmetric/Hermitian positive semidefinite. Define $\Abf\{\mathcal{D}\}$ as the sub-matrix of $\Abf$ obtained by choosing the rows that belong to the index set $\mathcal{D}$.
% Problem Formulation=====================================
\section{Problem Formulation}
A power network can be modeled as a directed graph $\Hcal\!\!=\!(\Vcal,\Ecal)$, with $\Vcal$ and $\Ecal$ as the set of buses and lines, respectively. For each bus $k\!\in\!\Vcal$, the demand forecast is denoted by $d_k\!\in\!\mathbb{C}$, whose real and imaginary parts account for active and reactive demands, respectively. Define $v_k\!\in\!\mathbb{C}$ as the complex voltage at bus $k$. Let $\Gcal$ be the set of generating units, each located at one of the buses. For each generating unit $g\in\Gcal$, the values $p_g$ and $q_g$, respectively, denote the amount of active and reactive powers. The unit incidence matrix $\Cbf\!\in\!\{0,1\}^{|\Gcal|\times|\Vcal|}$ is defined as a binary matrix whose $\!(g,k)\!$ entry is equal to one, if and only if the generating unit $g$ belongs to bus $k$. Additionally,
define the pair of matrices $\vec{\Cbf},\cev{\Cbf}\!\in\!\{0,1\}^{|\Ecal|\times|\Vcal|}$ as the \textit{from} and \textit{to} incidence matrices, respectively. The $(l,k)$ entry of $\vec{\Cbf}$ is equal to one, if and only if the line $l\!\in\!\Ecal$ starts at bus $k$, while the $(l,k)$ entry of $\cev{\Cbf}$ equals one, if and only if line $l$ ends at bus $k$. Define $\Ybf\!\in\!\Cbb^{|\Vcal|\times|\Vcal|}$ as the nodal admittance matrix of the network and $\vec{\Ybf}\!,{\cev{\Ybf}}\!\in\!\Cbb^{|\Ecal|\times|\Vcal|}$ as the \textit{from} and \textit{to} branch admittance matrices. Define $\ybf_{\shrm}\! =\! \gbf_{\shrm} \!+\! \irm \bbf_{\shrm} \!\in\!\Cbb^{|\Vcal|}$ as the vector of shunt admittances whose real and imaginary parts correspond to the shunt conductances and susceptances, respectively. The OPF problem can be formulated as,
%============================================================
\begin{subequations}\label{eq:OPF}
\begin{align}
& \underset{
\begin{subarray}{c} \!\!\pbf,\qbf\,\in\,\Rbb^{|\Gcal|}\\
             \;\:\,\vbf\,\in\,\Cbb^{|\Vcal|}\\
             \vec{\sbf},\cev{\sbf}\,\in \,\Cbb^{|\Ecal|}
             \end{subarray}
}{\text{minimize~~~}}
		&&\hspace{-0.5cm} h(\pbf)   \label{eq:OPF_obj}\\
& \text{subject to~~~}
		&&\hspace{-0.5cm}\dbf+\mathrm{diag}\{\vbf\;\!\vbf^{\ast}\Ybf^{\ast}\}=\Cbf^{\top}(\pbf+\irm\qbf)\label{eq:OPF_cons_1}\\
        &&&\hspace{-0.5cm}\mathrm{diag}\{\vec{\Cbf}\;\!\vbf\;\!\vbf^{\ast}\vec{\Ybf}^{\ast}\} = \vec{\sbf}\label{eq:OPF_cons_2}\\
		&&&\hspace{-0.5cm}\mathrm{diag}\{\cev{\Cbf}\;\!\vbf\;\!\vbf^{\ast}\cev{\Ybf}^{\ast}\} = \cev{\sbf}\label{eq:OPF_cons_3}\\
        &&&\hspace{-0.5cm}\vbf^2_{\mathrm{min}}\leq\lvert\vbf\rvert^2\leq \vbf^2_{\mathrm{max}}\label{eq:OPF_cons_4} \\
		&&&\hspace{-0.5cm}{\pbf_{\mathrm{min}}} \leq  \:\:\pbf\:\: \leq {\pbf_{\mathrm{max}}}\label{eq:OPF_cons_5}\\ 
		&&&\hspace{-0.5cm}{\qbf_{\mathrm{min}}} \leq  \:\:\qbf\:\: \leq {\qbf_{\mathrm{max}}}\label{eq:OPF_cons_6}\\
		&&&\hspace{-0.5cm}|\vec{\sbf}|^2 \leq \fbf^2_{\mathrm{max}}\label{eq:OPF_cons_7}\\
		&&&\hspace{-0.5cm}|\cev{\sbf}|^2 \leq \fbf^2_{\mathrm{max}}\label{eq:OPF_cons_8}%\\
\end{align}
\end{subequations}
%==========================================================
where $h(\pbf)\triangleq\cbf_{0}^{\top}\onebf+\cbf_{1}^{\top}\pbf+\pbf^{\top}[\cbf_{2}]\;\!\pbf$ is the objective function, $\cbf_{0},\cbf_{1},\cbf_{2} \in\Rbb^{|\Gcal|}_{+}$ are the vectors of fixed, linear, and quadratic cost coefficients, respectively.  %generating units $\Gcal$. 
Constraint \eqref{eq:OPF_cons_1} is the power balance equation, which accounts for conservation of energy at all buses of the network. Constraint \eqref{eq:OPF_cons_4} ensures that voltage magnitudes remain within pre-specified ranges, given by vectors $\vbf_{\mathrm{min}}, \vbf_{\mathrm{max}}\in\Rbb^{|\Vcal|}$. Power generation vectors are bounded by ${\pbf_{\mathrm{min}}},{\pbf_{\mathrm{max}}}\in\Rbb^{|\Gcal|}$ for active power and ${\qbf_{\mathrm{min}}},{\qbf_{\mathrm{max}}}\in\Rbb^{|\Gcal|}$ for reactive power. The flow of power entering the lines of the network from their starting and ending buses are denoted by $\vec\sbf\in\Cbb^{|\Ecal|}$ and $\cev\sbf\in\Cbb^{|\Ecal|}$, respectively, and upper bounded by the vector of thermal limits $\fbf_{\max}\in\Rbb^{|\Ecal|}$. 
%%%%%%%%%%%%%%%%%%%%%
\section{Preliminaries and Sensitivity Analysis}
%%%%%%%%%%%%%%%%
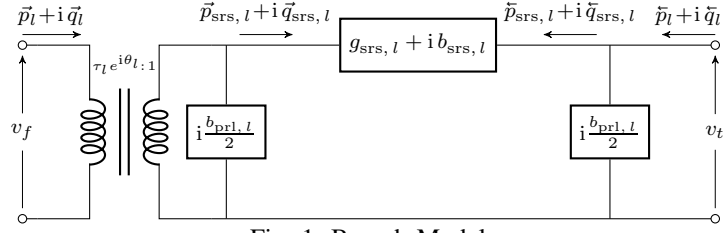
\begin{figure*}[!t]
  \begin{center}
  \vspace{2mm}
    \tikzset{R_F/.style = {draw,thick,minimum size=2em}}
    \begin{circuitikz}[american, cute inductors]
    %=============Transformer
		\draw (0,0) node [scale = 1.1, transformer core] (T){}  (T.base) node[below] {\tiny$\tau_l e^{\irm\theta_l}\!\!:\!1$};
        \draw (T.A1) to [short, -o] ++(-0.2,0) coordinate (TA1);
	    \draw (T.A2) to [short, -o] ++(-0.2,0) coordinate (TA2);
  	%============v_f
	\path (TA2) -- node (A2A1) {\footnotesize $v_{f}$} (TA1);
		\draw[->,>=stealth',shorten >=4pt, shorten <=0.5pt, very thin] (A2A1) edge[bend left=0]  (TA1);
        \draw[-,shorten >=0.5pt, shorten <=4pt, very thin] (TA2) edge[bend left=0]  (A2A1);
% 	%============v_f/N
% 	\path ([xshift=-8pt]T.B2) -- node (B2B1) {\large $\frac{v_f}{N}$} ([xshift=-8pt]T.B1);
%         \draw[->,>=stealth',shorten >=4pt, shorten <=1pt,very thin] (B2B1) edge[bend left=0]  ([xshift=-8pt]T.B1);
%         \draw[-,shorten >=1pt, shorten <=4pt, very thin] ([xshift=-8pt]T.B2) edge[bend left=0]  (B2B1);
%============p_vec+iq_vec
    \draw[->,>=stealth',very thin] ([xshift=0cm,yshift=4pt]TA1) -- node[above] {\footnotesize $\vec{p}_{l}\!+\!\irm\,\vec{q}_{l}$} +(22pt,0);
%============N*i_f
%   \draw[->,>=stealth',very thin] ([xshift=-0.65cm,yshift=4pt]T.B1) -- node[above] {\footnotesize $N\!\!\times\!i_f$} +(20pt,0);
    %=============
    \draw (T.B1) to [short] ++(0.2,0) coordinate (B1E);
    \draw (T.B2) to [short] ++(0.2,0) coordinate (B2E);
    
	%============== R_1 & below
    \path (B1E) -- node[R_F] (r1) {\footnotesize${
g_{\sm{\mathrm{srs}}{6},\,l} +\irm\,
b_{\sm{\mathrm{srs}}{6},\,l}}$} ++(5.1,0) coordinate (E1Y);
        	\draw[-] (r1) edge[bend left=0]  (E1Y);
        	\draw[-] (B1E) edge[bend left=0]  (r1);
            
    \draw (B2E) to [short] ++(5.1,0) coordinate (E2Y);
    %============p_s_vec+iq_s_vec
    \draw[->,>=stealth',very thin] ([xshift=5pt,yshift=4pt]B1E) -- node[above] {\footnotesize ${\vec{p}_{\sm{\mathrm{srs}}{6},\,l}\!+\! \irm \,\vec{q}_{\sm{\mathrm{srs}}{6},\,l}}$} +(20pt,0);
     \draw[->,>=stealth',very thin] ([xshift=-5pt,yshift=4pt]E1Y) -- node[above] {\footnotesize ${\cev{p}_{\sm{\mathrm{srs}}{6},\,l}\!+\! \irm \, \cev{q}_{\sm{\mathrm{srs}}{6},\,l}}$} +(-20pt,0);
    %============== End lines
    \draw (E1Y) to [short,-o] ++(1.4,0) coordinate (Y1E);
    \draw (E2Y) to [short,-o] ++(1.4,0) coordinate (Y2E);
    
  	%============== R_2
	\path (B1E) -- node[R_F] (r2) {\footnotesize$\irm\frac{b_{\mathrm{prl},\,l}}{2}$} (B2E);
        	\draw[-] (r2) edge[bend left=0]  (B2E);
        	\draw[-] (B1E) edge[bend left=0]  (r2);
            
	%============== R_3
	\path (E1Y) -- node[R_F] (r3) {\footnotesize$\irm\frac{b_{\mathrm{prl},\,l}}{2}$} (E2Y);
        	\draw[-] (r3) edge[bend left=0]  (E2Y);
        	\draw[-] (E1Y) edge[bend left=0]  (r3);
            
	%============== p_cev+q_cev
    \draw[->,>=stealth',very thin] ([xshift=0cm,yshift=4pt]Y1E) -- node[above] {\footnotesize $\cev{p}_l\!+\! \irm \,\cev{q}_l$} +(-20pt,0);
	%============== v_t
	\path (Y2E) -- node (EY1Y2) {\footnotesize $v_{t}$} (Y1E);
        	\draw[->,>=stealth',shorten >=4pt, shorten <=1pt, very thin] (EY1Y2) edge[bend left=0]  (Y1E);
        	\draw[-,shorten >=1pt, shorten <=4pt, very thin] (Y2E) edge[bend left=0]  (EY1Y2);

   \end{circuitikz} 
  \end{center}
    \vspace{-4mm}
  \caption{Branch Model}
  \vspace{-4mm}
  \label{fig:branch}
\end{figure*}
%%%%%%%%%%%%%%%%
Consider the standard $\pi$-model of line $l=(f,t)\in\mathcal{E}$, with series admittance $y_{\sm{\mathrm{srs},\,l}{6}}\triangleq g_{\sm{\mathrm{srs},\,l}{6}}+\irm\,{b_{\sm{\mathrm{srs},\,l}{6}}}$ and total shunt susceptance $b_{\sm{\mathrm{prl},\,l}{6}}$, in series with a phase shifting transformer whose tap ratio has magnitude $\tau_l$ and phase shift angle $\theta_l$ \cite{zimmerman2011matpower}. The model is shown in Figure \ref{fig:branch}. Define $\vbf_{l}\triangleq \begin{bmatrix}v_{f}, v_{t}\end{bmatrix}^{\!\!\top}\!\!$ as the vector of the complex voltages at the two ends of the line $l$. The active and reactive power flows entering the line $l$ through the \textit{from} and \textit{to} ends of the branch are equal to, 
\begin{subequations} 
\begin{align}
&\vec{p}_{l} = \vec{p}_{\mathrm{srs},\,l},
&& \vec{q}_{l} = \vec{q}_{\mathrm{srs},\,l}+
\frac{b_{\mathrm{prl},\;l}}{2\tau^2_l}|v_f|^2,\label{eq:series_fpower}
\\
&\cev{p}_{l} = \cev{p}_{\mathrm{srs},\,l}, 
&& \cev{q}_{l} = \cev{q}_{\mathrm{srs},\,l}+
\frac{b_{\mathrm{prl},\;l}}{2}|v_t|^2,\label{eq:series_tpower}
\end{align}
\end{subequations}
where
$\vec{p}_{\mathrm{srs},\,l}+\irm\,\vec{q}_{\mathrm{srs},\,l}$ and $\cev{p}_{\mathrm{srs},\,l}+\irm\,\cev{q}_{\mathrm{srs},\,l}$ are complex powers passing through the series element from the two ends. Additionally,
\begin{subequations} 
\begin{align}
&\vec{p}_{\mathrm{srs},\,l} = \vbf_{l}^{\ast}\,\vec{\Ybf}_{\mathrm{p};\,l}^{\phantom{\ast}}\,\vbf_{l}^{\phantom{\ast}},
&& \vec{q}_{\mathrm{srs},\,l} = \vbf_{l}^{\ast}\,\vec{\Ybf}_{\mathrm{q};\,l}^{\phantom{\ast}}\,\vbf_{l}^{\phantom{\ast}},\label{eq:series_fpower}
\\
&\cev{p}_{\mathrm{srs},\,l} = \vbf_{l}^{\ast}\,\cev{\Ybf}_{\mathrm{p};\,l}^{\phantom{\ast}}\,\vbf_{l}^{\phantom{\ast}}, 
&& \cev{q}_{\mathrm{srs},\,l} = \vbf_{l}^{\ast}\,\cev{\Ybf}_{\mathrm{q};\,l}^{\phantom{\ast}}\,\vbf_{l}^{\phantom{\ast}},\label{eq:series_tpower}
\end{align}
\end{subequations}
where, $\vec{\Ybf}_{\mathrm{p};\,l}$, 
$\vec{\Ybf}_{\mathrm{q};\,l}$,
$\cev{\Ybf}_{\mathrm{p};\,l}$, and $\cev{\Ybf}_{\mathrm{q};\,l}$ are given as
\begin{equation}\notag%\small
\begin{aligned}
&\!\vec{\Ybf}_{\mathrm{p};\,l}\!\triangleq\!\!\begin{bmatrix} \frac{g_{\sm{\mathrm{srs}}{6},\,l}}{\tau_l^2} & 
\!\!\frac{e^{\irm\theta_l}\,y_{\sm{\mathrm{srs}}{6},\,l}}{\smallMinus 2\tau_l}\! \\[1.0ex] 
\frac{\!y_{\sm{\mathrm{srs}}{6},\,l}^{\ast}}{\smallMinus2\tau_l e^{\irm\theta_l} } & 0 \\\vspace{-2.5mm}\end{bmatrix}\!\!,
%%%
&&\vec{\Ybf}_{\mathrm{q};\,l}\!\triangleq\!\!\begin{bmatrix} \frac{b_{\sm{\mathrm{srs}}{6},\,l}}{\smallMinus\tau_l^2} & \!\!\frac{ e^{\irm\theta_l}\,y_{\sm{\mathrm{srs}}{6},\,l}}{2\tau_l\irm} \\[1.0ex] \frac{ y_{\sm{\mathrm{srs}}{6},\,l}^{\ast}}{\smallMinus 2\tau_l\irm e^{\irm\theta_l}} & 0
\\\vspace{-2.5mm}\end{bmatrix}\!\!,\\
%%%
&\!\cev{\Ybf}_{\mathrm{p};\,l}\!\triangleq\!\!\begin{bmatrix} 0 &  
\frac{e^{\irm\theta_l}\,y_{\sm{\mathrm{srs}}{6},\,l}^{\ast}}{\smallMinus 2\tau_l} \\[1.0ex] \frac{ y_{\sm{\mathrm{srs}}{6},\,l}}{\smallMinus2\tau_l e^{\irm\theta_l}}\!\!\! & g_{\sm{\mathrm{srs}}{6},\,l}\\\vspace{-2.5mm}\end{bmatrix}\!\!, 
%%%
&&\cev{\Ybf}_{\mathrm{q};\,l}\!\triangleq\!\!\begin{bmatrix} 0 & \frac{ e^{\irm\theta_l}\,y_{\sm{\mathrm{srs}}{6},\,l}^{\ast}}{-2\tau_l\irm} \\[1.0ex]
\frac{ y_{\sm{\mathrm{srs}}{6},\,l}}{2\tau_l\irm e^{\irm\theta_l}} & \smallMinus b_{\sm{\mathrm{srs}}{6},\,l}\\\vspace{-2.5mm}\end{bmatrix}\!\!.
\end{aligned}
\end{equation}
%
%%%%%%%%%%%%%%%%%%%%%%%%%%%%%%%%%%%%%%%%%%%%%% Sensitivity
%where $y_{\sm{\mathrm{srs}}{6}}\triangleq g_{\sm{\mathrm{srs}}{6}}+\irm\,{b_{\sm{\mathrm{srs}}{6}}}$. 

The next definition introduces the notion of {\it sensitivity measure} for power systems, which will be used later in the paper. 
\begin{definition}\label{def:sens}
\vspace{1mm}
The sensitivity measure of the power system under study is defined as 
\begin{align}
&\!\!\!P\triangleq  2|\mathcal{N}|+2|\mathcal{L}|+ \|\mathbf{y}_{\mathrm{sh}}\|_2
+\sum_{l\in\mathcal{E}}\big({\frac{|b_{\mathrm{prl},\;l}|}{2\tau_l^2}+\frac{|b_{\mathrm{prl},\;l}|}{2}}\big)
\nonumber\\
&\!\!\!+\sqrt{2}\sum_{l\in\mathcal{E}}\big({
\|\vec{\Ybf}_{\mathrm{p};\,l}\|_1+
\|\cev{\Ybf}_{\mathrm{p};\,l}\|_1+
\|\vec{\Ybf}_{\mathrm{q};\,l}\|_1+
\|\cev{\Ybf}_{\mathrm{q};\,l}\|_1}\big).\!\!\!\!\!
\end{align}
% $P \triangleq \max \{ 1, 
% g_{\mathrm{sh}}^{\max}, 
% b_{\mathrm{sh}}^{\max}$, 
% $\vec{y}_{\mathrm{p}}^{\;\max}, 
% \cev{y}_{\mathrm{p}}^{\;\max},
% \vec{y}_{\mathrm{q}}^{\;\max}, \cev{y}_{\mathrm{q}}^{\;\max}\}$,
% where $g_{\mathrm{sh}}^{\max} \triangleq \max\{\gbf_{\shrm}\}$, $b_{\mathrm{sh}}^{\max} \triangleq \max \{\bbf_{\shrm}\}$, and %${y}_{\vec{p}_{b}}$, ${y}_{\vec{q}_{b}}$, ${y}_{\cev{p}_{b}}$, ${y}_{\cev{q}_{b}}$ 
% \begin{equation}
% \small
% \begin{aligned}\notag
% \vspace{-.1mm}
% &\vec{y}_{\mathrm{p}}^{\;\max}\triangleq\max_{l\in\Ecal} \sigma_{\maxrm}(\vec{\Ybf}_{\mathrm{p};\,l}), & \!\!\!\vec{y}_{\mathrm{q}}^{\;\max}\triangleq\max_{l\in\Ecal} \{\sigma_{\maxrm}(\vec{\Ybf}_{\mathrm{q};\,l})+\frac{|b_{\mathrm{prl},\;l}|}{2\tau_l^2}\},\\
% &
% \cev{y}_{\mathrm{p}}^{\;\max}\triangleq\max_{l\in\Ecal} \sigma_{\maxrm}(\cev{\Ybf}_{\mathrm{p};\,l}), & \!\!\!\cev{y}_{\mathrm{q}}^{\;\max}\triangleq\max_{l\in\Ecal}\{\sigma_{\maxrm}(\cev{\Ybf}_{\mathrm{q};\,l})+\frac{|b_{\mathrm{prl},\;l}|}{2}\}.\\
% \end{aligned}
% \end{equation}
% \vspace{.05mm}
\end{definition}
%%%%%%%%%%%%%%%%%%%%%%%%%%%%%%%%%%%%%%%%%%%%%%%%%%%%
To derive the optimality conditions %for an arbitrary point of
of the problem \cref{eq:OPF_obj,eq:OPF_cons_1,eq:OPF_cons_2,eq:OPF_cons_3,eq:OPF_cons_4,eq:OPF_cons_5,eq:OPF_cons_6,eq:OPF_cons_7,eq:OPF_cons_8} we define the Jacobian of equality and inequality constraints.
%%%%%%%%%%%%%%%%%%%%%%%%%%%%%%%%%%%%%%%%%%%%%%%%%%%%%%%%%%%%%%%%%%%%%%%%%%%%%%%%%%%%%%%%%
\begin{definition}
For every arbitrary point $\!\xbf\!=\!(\vbf,\pbf\!+\!\irm\qbf,\vec{\sbf},\cev{\sbf})\in\Cbb^{|\Vcal|}\times\Cbb^{|\Gcal|}	\times\Cbb^{|\Ecal|}\times\Cbb^{|\Ecal|}$, the Jacobian of equality constraints \cref{eq:OPF_cons_1,eq:OPF_cons_2,eq:OPF_cons_3} is equal to $\Jbf^{\sm{=}{5}} = \realrm\{\hat{\Jbf}^{\sm{=}{5}}\}$, where
	%%%%%%%%% Still Def: eq
%   
%\vspace{-5mm}    
	\begin{equation}\notag\footnotesize
	\begin{aligned}
	\hat{\Jbf}^{\sm{=}{5}} \!\triangleq\!\!
	\begin{bmatrix}
	\phantom{\smallMinus} 2\,[[\gbf_{\shrm}]\vbf] &\hspace{-0.1cm} \!\!\smallMinus 2\irm\,[[\gbf_{\shrm}]\vbf] &\hspace{-0.2cm} \smallMinus \Cbf^{\top} &\hspace{-0.2cm} \mathbf{0} &\hspace{-0.2cm} \vec\Cbf^{\top} &\hspace{-0.2cm} \mathbf{0} &\hspace{-0.2cm} \cev\Cbf^{\top} &\hspace{-0.2cm} \mathbf{0}\\[0.5ex]
	\smallMinus 2\,[[\bbf_{\shrm}]\vbf] &\hspace{-0.1cm} \phantom{\smallMinus}\!\!2\irm\,[[\bbf_{\shrm}]\vbf] &\hspace{-0.2cm} \mathbf{0} &\hspace{-0.2cm} \smallMinus \Cbf^{\top} &\hspace{-0.2cm} \mathbf{0} &\hspace{-0.2cm} \vec\Cbf^{\top} &\hspace{-0.2cm} \mathbf{0} &\hspace{-0.2cm} \cev\Cbf^{\top}\\[0.5ex]
	2\phantom{\irm}\vec{\Ubf}_{\!1} &\hspace{-0.1cm} \smallMinus2\phantom{\irm} \vec{\Ubf}_{\!2} &\hspace{-0.2cm} \mathbf{0} & \hspace{-0.2cm} \mathbf{0} &\hspace{-0.2cm} \smallMinus\Ibf_{|\Ecal|} &\hspace{-0.2cm} \mathbf{0} &\hspace{-0.2cm} \mathbf{0} &\hspace{-0.2cm} \mathbf{0}\\[0.5ex]
	2\irm \vec{\Ubf}_{\!1} &\hspace{-0.1cm} \smallMinus2\irm \vec{\Ubf}_{\!2} &\hspace{-0.2cm} \mathbf{0} &\hspace{-0.2cm} \mathbf{0} &\hspace{-0.2cm} \mathbf{0} &\hspace{-0.2cm} \smallMinus\Ibf_{|\Ecal|} &\hspace{-0.2cm} \mathbf{0} &\hspace{-0.2cm} \mathbf{0}\\[0.5ex]
	2\phantom{\irm}\cev{\Ubf}_{\!1} &\hspace{-0.1cm}  \smallMinus2 \phantom{\irm}\cev{\Ubf}_{\!2} &\hspace{-0.2cm} \mathbf{0} &\hspace{-0.2cm} \mathbf{0} &\hspace{-0.2cm} \mathbf{0} &\hspace{-0.2cm} \mathbf{0} & \hspace{-0.2cm}\smallMinus\Ibf_{|\Ecal|} &\hspace{-0.2cm} \mathbf{0}\\[0.5ex]
	2\irm\cev{\Ubf}_{\!1} &\hspace{-0.1cm}  \smallMinus2\irm \cev{\Ubf}_{\!2} &\hspace{-0.2cm} \mathbf{0} &\hspace{-0.2cm} \mathbf{0} &\hspace{-0.2cm} \mathbf{0} &\hspace{-0.2cm} \mathbf{0} &\hspace{-0.2cm} \mathbf{0} &\hspace{-0.2cm} \smallMinus\Ibf_{|\Ecal|}
	\end{bmatrix}
	\end{aligned}
	\end{equation}
%
%\vspace{-1mm}
%\noindent
and matrices $\vec{\Ubf}_{1}$, $\vec{\Ubf}_{2}$, $\cev{\Ubf}_{1}$, and $\cev{\Ubf}_{2}$ are defined as
%
%\vspace{-5mm}
\begin{equation}\notag\small
\begin{aligned} 
\vec{\Ubf}_{1} &\!\triangleq\!\frac{1}{2}([\vbf^{\ast} \vec\Cbf^{\top}\!]\vec{\Ybf}\!+\![\vec\Ybf \vbf]\vec\Cbf), & \vec{\Ubf}_{2} &\!\triangleq\!\frac{1}{2\irm}([\vbf^{\ast} \vec\Cbf^{\top}\!]\vec{\Ybf}\!-\![\vec\Ybf \vbf]\vec\Cbf),\\
\cev{\Ubf}_{1} &\!\triangleq\! \frac{1}{2}([\vbf^{\ast} \cev\Cbf^{\top}\!]\cev{\Ybf}\!+\![\cev\Ybf \vbf]\cev\Cbf), & \cev{\Ubf}_{2} &\!\triangleq\!\frac{1}{2\irm}([\vbf^{\ast} \cev\Cbf^{\top}\!]\cev{\Ybf}\!-\![\cev\Ybf \vbf]\cev\Cbf).
\end{aligned}
\end{equation}
%
%\vspace{-1mm}
%\noindent
Moreover, the Jacobian of inequality constraints \cref{eq:OPF_cons_4,eq:OPF_cons_5,eq:OPF_cons_6,eq:OPF_cons_7,eq:OPF_cons_8} are, respectively, given as
%
%\vspace{-5mm}
	\begin{subequations}
	\begin{align}
	\hspace{-4mm}\Jbf^{\sm{\leq}{5}}_1\!&\triangleq\!
	2\mathrm{real}\{\!\begin{bmatrix}
	[\vbf]&\hspace{-0.0cm} -\irm[\vbf]&\hspace{-0.0cm} \mathbf{0}_{|\mathcal{V}|\times(2|\mathcal{G}|+4|\Ecal|)}
	\end{bmatrix}\!\},\\
	\hspace{-4mm}\Jbf^{\sm{\leq}{5}}_2\!&\triangleq\!
	\begin{bmatrix}
	\mathbf{0}_{|\Gcal|\times(2|\mathcal{V}|)} 
	&\hspace{-0.0cm} \, \Ibf_{|\Gcal|}& \mathbf{0}_{|\Gcal|\times(|\mathcal{G}|+4|\Ecal|)}
	\end{bmatrix},\\
	\hspace{-4mm}\Jbf^{\sm{\leq}{5}}_3\!&\triangleq\!
	\begin{bmatrix}
	\mathbf{0}_{|\Gcal|\times(2|\mathcal{V}|+|\mathcal{G}|)} 
	&\hspace{-0.0cm} \, \Ibf_{|\Gcal|}& \mathbf{0}_{|\Gcal|\times(4|\Ecal|)}
	\end{bmatrix},\\
	\hspace{-4mm}\vec{\Jbf}^{\sm{\leq}{5}}_4\!&\triangleq\!
	2\mathrm{real}\{\!\begin{bmatrix}
	\mathbf{0}_{|\mathcal{E}|\times(2|\mathcal{V}|+2|\mathcal{G}|)}  &\hspace{-0.0cm} 
	[\cev\sbf] &\hspace{-0.1cm} -\irm[\cev\sbf] &\hspace{-0.1cm} \mathbf{0}_{|\mathcal{E}|\times(2|\Ecal|)}
	\end{bmatrix}\!\},\!\!\!\!\!\!	\\
	\hspace{-4mm}\cev{\Jbf}^{\sm{\leq}{5}}_4\!&\triangleq\!
	2\mathrm{real}\{\!\begin{bmatrix}
\mathbf{0}_{|\mathcal{E}|\times(2|\mathcal{V}|+2|\mathcal{G}|+2|\Ecal|)} &\hspace{-0.0cm} [\cev\sbf] &\hspace{-0.0cm} -\irm[\cev\sbf]
	\end{bmatrix}\!\}.
	\end{align}
\end{subequations}
\vspace{.05mm}
\end{definition}
%%%%%%%%%%%%%%%%%%%%%%%%%%%%%%%%%%%%%%%%%%%%%%%%%
Given a feasible solution and its Jacobian, the well-known linear independence constraint qualification (LICQ) condition is used to characterize well-behaved feasible points.
%%%%%%%%%%%%%%%%%%%%%%%%%%%%%%%%%%%%%%%%%%%%%%%%%%
\begin{definition}[LICQ] \label{def:LICQ}
\vspace{1mm}
Consider a feasible point $(\vbf,\pbf\!+\!\irm\qbf,\vec{\sbf},\cev{\sbf})
\in\Cbb^{|\Vcal|}\!\times\!\Cbb^{|\Gcal|}
\!\times\!\Cbb^{|\Ecal|}\!\times\!\Cbb^{|\Ecal|}$ for the problem \cref{eq:OPF_obj,eq:OPF_cons_1,eq:OPF_cons_2,eq:OPF_cons_3,eq:OPF_cons_4,eq:OPF_cons_5,eq:OPF_cons_6,eq:OPF_cons_7,eq:OPF_cons_8}. The point $(\vbf,\pbf+\irm\qbf,\vec{\sbf},\cev{\sbf})$ is said to satisfy the LICQ condition if the gradient vectors of equality constraints \cref{eq:OPF_cons_1,eq:OPF_cons_2,eq:OPF_cons_3} and those inequality constraints \cref{eq:OPF_cons_4,eq:OPF_cons_5,eq:OPF_cons_6,eq:OPF_cons_7,eq:OPF_cons_8} that are active form a linearly independent set. In other words, the LICQ condition holds, if the matrix
%$\Jbf_{\mathcal{B}^{\downarrow}_1,\mathcal{B}^{\uparro%w}_1,
%\mathcal{B}^{\downarrow}_2,\mathcal{B}^{\uparrow}_2, 
%\mathcal{B}^{\downarrow}_3,\mathcal{B}^{\uparrow}_3,
%\vec{\mathcal{B}}_4,\cev{\mathcal{B}}_4}(\xbf)$ 
\begin{align}\label{jacob}\small
 &\!\!\!\Jbf_{\mathcal{B}^{\downarrow}_1,\mathcal{B}^{\uparrow}_1,
\mathcal{B}^{\downarrow}_2,\mathcal{B}^{\uparrow}_2, 
\mathcal{B}^{\downarrow}_3,\mathcal{B}^{\uparrow}_3,
\vec{\mathcal{B}}_4,\cev{\mathcal{B}}_4}
(\xbf) =\!
\big[({\Jbf^{\sm{=}{5}}})^{\!\top},
\;\Jbf^{\sm{\leq}{5}}_1\{\mathcal{B}^{\downarrow}_1\cup\mathcal{B}^{\uparrow}_1\}^{\!\top},\nonumber\\
 &\Jbf^{\sm{\leq}{5}}_2\{ \mathcal{B}^{\downarrow}_2\cup\mathcal{B}^{\uparrow}_2\}^{\!\top},
  \;\Jbf^{\sm{\leq}{5}}_3\{ \mathcal{B}^{\downarrow}_3\cup\mathcal{B}^{\uparrow}_3\}^{\!\top},
 \;\vec{\Jbf}^{\sm{\leq}{5}}_4\{
  \vec{\mathcal{B}}_4\}^{\!\top},
  \;\cev{\Jbf}^{\sm{\leq}{5}}_4\{
  \cev{\mathcal{B}}_4\}^{\!\top}\big]^{\!\top}\!\!\!
 \end{align}
is full row rank, where
\begin{equation}\notag\small
\begin{aligned}
&\mathcal{B}^{\downarrow}_1=
\{k\in\mathcal{V} \,|\, |v_k| = v_{\minrm\phantom{\maxrm}\!\!\!\!\!\!\!\!,\,k}
\},
&& \mathcal{B}^{\uparrow}_1=
\{k\in\mathcal{V} \,|\, |v_k| = v_{\maxrm\phantom{\minrm}\!\!\!\!\!\!\!\!,\,k}
\},\\
&\mathcal{B}^{\downarrow}_2=
\{g\in\mathcal{G} \,|\, p_g = p_{\minrm\phantom{\maxrm}\!\!\!\!\!\!\!,\,g}\},
&&\mathcal{B}^{\uparrow}_2=
\{g\in\mathcal{G} \,|\, p_g = p_{\maxrm\phantom{\minrm}\!\!\!\!\!\!\!,\,g}\},\\
&\mathcal{B}^{\downarrow}_3=
\{g\in\mathcal{G} \,|\, q_g = q_{\minrm\phantom{\maxrm}\!\!\!\!\!\!\!,\,g}\},
&&\mathcal{B}^{\uparrow}_3=
\{g\in\mathcal{G} \,|\, q_g = q_{\maxrm\phantom{\minrm}\!\!\!\!\!\!\!,\,g}\},\\
&\vec{\mathcal{B}}_4=
\{\phantom{g}\!\!l\in\mathcal{E} \,|\, |\vec{s}_l| = f_{\maxrm,\,l}\},
&&\cev{\mathcal{B}}_4=
\{\phantom{g}\!\!l\in\mathcal{E} \,|\, |\cev{s}_l| = f_{\maxrm,\,l}\}.
\end{aligned}
\end{equation}
\vspace{0.05mm}
%
%\vspace{-7mm}
% %
% %\vspace{-1mm}
% %\noindent
% is full row rank, where
% $\mathcal{B}^{\downarrow}_1$, $\mathcal{B}^{\uparrow}_1$, 
% $\mathcal{B}^{\downarrow}_2$, $\mathcal{B}^{\uparrow}_2$, 
% $\mathcal{B}^{\downarrow}_3$, $\mathcal{B}^{\uparrow}_3$,
% $\mathcal{B}^{\uparrow}_4$, and $\mathcal{B}^{\uparrow}_5$ are the sets of indices, corresponding to the zero elements of the vectors
% %
% %\vspace{-6mm}
% \begin{align}
% |\vbf|^2 -\vbf^2_{\maxrm},&&
% \vbf^2_{\minrm}-|\vbf|^2,&&
% \pbf - \pbf_{\maxrm}^{\phantom{2}},&&
% \pbf_{\minrm}^{\phantom{2}} - \pbf,\nonumber\\
% \qbf - \qbf_{\maxrm}^{\phantom{2}},&& 
% \qbf_{\minrm}^{\phantom{2}}-\qbf,&&
% |\vec{\sbf}|^2 - \fbf^2_{\maxrm},&&
% |\cev{\sbf}|^2 - \fbf^2_{\maxrm},\nonumber
% \end{align}
% %\vspace{1mm}
% %
% %\vspace{-2mm}
% %\noindent
% respectively.
\end{definition}
%%%%%%%%%%%%%%%%%%%%%%%%%%%%%%%%%%%%%%%%%%%%%%%%%%%%%%%%%%%%%%%%%%%%%%%%%%%%%%%%%%%%%%%%%

Since the LICQ condition is only defined for feasible points, in this paper, we introduce a generalization of the LICQ condition that is applicable to infeasible points as well. To this end, we first need a measure for the distance between an arbitrary point $\xbf_{\sm{0}{5}}=(\vbf_{\sm{0}{5}},\pbf_{\sm{0}{5}}\!+\!\irm\qbf_{\sm{0}{5}},\vec{\sbf}_{\sm{0}{5}},\cev{\sbf}_{\sm{0}{5}})
	\in\Cbb^{|\Vcal|}\!\times\!\Cbb^{|\Gcal|}
	\!\times\!\Cbb^{|\Ecal|}\!\times\!\Cbb^{|\Ecal|}$ and the feasible set of the OPF problem \cref{eq:OPF_obj,eq:OPF_cons_1,eq:OPF_cons_2,eq:OPF_cons_3,eq:OPF_cons_4,eq:OPF_cons_5,eq:OPF_cons_6,eq:OPF_cons_7,eq:OPF_cons_8}.
%%%%%%%%%%%%%%%%%%%%%%%%%%%%%%%%%%%%%%%%%%%%%%%%%%%%%%%%%%%%%%%%%%%%%%%%%%%%%%%%%%%%%%%%%
\begin{definition}[Feasibility distance]\label{def:feasdist}
\vspace{1mm}
Define the OPF feasible set 
%$\Fcal\subset \mathbb{C}^{|\Vcal|+|\Gcal|+2|\Ecal|}$ 
$\Fcal\subset\Cbb^{|\Vcal|}\!\times\!\Cbb^{|\Gcal|}
\!\times\!\Cbb^{|\Ecal|}\!\times\!\Cbb^{|\Ecal|}$
as the set of all $\xbf\!=\!(\vbf,\pbf+\irm\qbf,\vec{\sbf},\cev{\sbf})$ that satisfy \cref{eq:OPF_cons_1,eq:OPF_cons_2,eq:OPF_cons_3,eq:OPF_cons_4,eq:OPF_cons_5,eq:OPF_cons_6,eq:OPF_cons_7,eq:OPF_cons_8}. Moreover, for every arbitrary $\xbf_{\sm{0}{5}}\!=\!(\vbf_{\sm{0}{5}},\sbf_{\sm{0}{5}},\vec{\sbf}_{\sm{0}{5}},\cev{\sbf}_{\sm{0}{5}})
\!\in\Cbb^{|\Vcal|}\!\times\!\Cbb^{|\Gcal|}	\!\times\!\Cbb^{|\Ecal|}\!\times\!\Cbb^{|\Ecal|}$, define the feasibility distance $\delta_{\Mbf}(\xbf_{\sm{0}{5}})$ as
%\vspace{-6mm}
\begin{align}\small
\min_{\xbf\in\mathcal{F}}
\!\left(\|\vbf\!-\!\vbf_{\sm{0}{5}}\|_{\mathbf{M}}^2\!+\!
\|\pbf+\irm\qbf\!-\!\sbf_{\sm{0}{5}}\|^2_2\!+\!
\|\vec{\sbf}\!-\!\vec{\sbf}_{\sm{0}{5}}\|^2_2\!+\!
\|\cev{\sbf}\!-\!\cev{\sbf}_{\sm{0}{5}}\|^2_2\right)^{\frac{1}{2}},\nonumber
\end{align}
%\vspace{-1mm}
%\noindent
where $\Mbf\in\Hbb_{|\mathcal{V}|}$ is arbitrary.
\end{definition}
%%%%%%%%%%%%%%%%%%%%%%%%%%%%%%%%%%%%%%%%%%%%%%%%%%%%%%%%%%%%%%%%%%%%%%%%%%%%%%%%%%%%%%%%%
\begin{definition}[Generalized LICQ] \label{def:GLICQ}
\vspace{1mm}
The point 
$\xbf=(\vbf,\pbf+\irm\qbf,\vec{\sbf},\cev{\sbf})
\in\Cbb^{|\Vcal|}\!\times\!\Cbb^{|\Gcal|}
\!\times\!\Cbb^{|\Ecal|}\!\times\!\Cbb^{|\Ecal|}$
%\begin{align}\small
%\Jbf(\xbf) =
%&\big[({\Jbf^{\sm{=}{5}}})^{\!\top},
% \;\;\Jbf^{\sm{\leq}{5}}_1\{\mathcal{B}^{\downarrow}_1\cup\mathcal{B}^{\uparrow}_1\}^{\!\top},
% \;\;\Jbf^{\sm{\leq}{5}}_2\{ \mathcal{B}^{\downarrow}_2\cup\mathcal{B}^{\uparrow}_2\}^{\!\top},\nonumber\\
% &\;\;\;\;\;\;\;\;\;\;\;\Jbf^{\sm{\leq}{5}}_3\{ \mathcal{B}^{\downarrow}_3\cup\mathcal{B}^{\uparrow}_3\}^{\!\top},
% \;\;\Jbf^{\sm{\leq}{5}}_4\{
% \mathcal{B}^{\uparrow}_4\}^{\!\top},
% \;\;\Jbf^{\sm{\leq}{5}}_5\{
% \mathcal{B}^{\uparrow}_5\}^{\!\top}\big]
%\end{align}
%and define
%$\mathcal{B}^{\downarrow}_1$, $\mathcal{B}^{\uparrow}_1$, 
%$\mathcal{B}^{\downarrow}_2$, $\mathcal{B}^{\uparrow}_2$, 
%$\mathcal{B}^{\downarrow}_3$, $\mathcal{B}^{\uparrow}_3$,
%$\mathcal{B}^{\uparrow}_4$, and $\mathcal{B}^{\uparrow}_5$ represent the sets of indices corresponding to the non-negative elements of vectors 
%\vspace{-6mm}
%The point $\xbf$ 
is said to satisfy the Generalized LICQ condition if the matrix
$\Jbf_{\mathcal{B}^{\downarrow}_1,\mathcal{B}^{\uparrow}_1,
\mathcal{B}^{\downarrow}_2,\mathcal{B}^{\uparrow}_2, 
\mathcal{B}^{\downarrow}_3,\mathcal{B}^{\uparrow}_3,
\vec{\mathcal{B}}_4,\cev{\mathcal{B}}_4}(\xbf)$ from the equation \eqref{jacob} is
full row rank, where
\begin{equation}\notag\small
\begin{aligned}
&\mathcal{B}^{\downarrow}_1=
\{k\in\mathcal{V} \,|\, -|v_k|^2 + v^2_{\minrm\phantom{\maxrm}\!\!\!\!\!\!\!\!,\,k} +
\delta_{\Mbf\!}(\xbf)^2 + 2\delta_{\Mbf\!}(\xbf)|v_k|\geq 0\},\\
& \mathcal{B}^{\uparrow}_1=
\{k\in\mathcal{V} \,|\, +|v_k|^2- v^2_{\maxrm\phantom{\minrm}\!\!\!\!\!\!\!\!,\,k}+
\delta_{\Mbf\!}(\xbf)^2 + 2\delta_{\Mbf\!}(\xbf)|v_k|\geq 0\},\\
&\mathcal{B}^{\downarrow}_2=
\{g\in\mathcal{G} \,|\, - p_g + p_{\minrm\phantom{\maxrm}\!\!\!\!\!\!\!,\,g}+\delta_{\Mbf\!}(\xbf)\geq 0\},\\
&\mathcal{B}^{\uparrow}_2=
\{g\in\mathcal{G} \,|\, + p_g - p_{\maxrm\phantom{\minrm}\!\!\!\!\!\!\!,\,g}+\delta_{\Mbf\!}(\xbf)\geq 0\},\\
&\mathcal{B}^{\downarrow}_3=
\{g\in\mathcal{G} \,|\, - q_g + q_{\minrm\phantom{\maxrm}\!\!\!\!\!\!\!,\,g}+\delta_{\Mbf\!}(\xbf)\geq 0\},\\
&\mathcal{B}^{\uparrow}_3=
\{g\in\mathcal{G} \,|\, + q_g - q_{\maxrm\phantom{\minrm}\!\!\!\!\!\!\!,\,g}+\delta_{\Mbf\!}(\xbf)\geq 0\},\\
&\vec{\mathcal{B}}_4=
\{\phantom{g}\!\!l\in\mathcal{E} \,|\, |\vec{s}_l|^2 - f^2_{\maxrm,\,l} +
\delta_{\Mbf\!}(\xbf)^2 + 2 \delta_{\Mbf\!}(\xbf)|\vec{s}_l|\geq 0\},\\
&\cev{\mathcal{B}}_4=
\{\phantom{g}\!\!l\in\mathcal{E} \,|\, |\cev{s}_l|^2 - f^2_{\maxrm,\,l} +
\delta_{\Mbf\!}(\xbf)^2 + 2 \delta_{\Mbf\!}(\xbf)|\cev{s}_l|\geq 0\}.
\end{aligned}
\end{equation}
% \begin{equation}\notag\small
% \begin{aligned}
% &|\vbf|^2\!\msh-\!\vbf^2_{\maxrm}\!\msh+\!
% \delta_{\Mbf\!}(\xbf)^2\!\msh+\!\msh 2 \delta_{\Mbf\!}(\xbf)|\vbf|,\\
% &\vbf^2_{\minrm}\!\msh-\msh|\vbf|^2\!\msh +\!
% \delta_{\Mbf\!}(\xbf)^2\!\msh +\!\msh 2\delta_{\Mbf\!}(\xbf)|\vbf|,\\
% &\pbf - \pbf_{\maxrm}+\delta_{\Mbf\!}(\xbf),\quad
% &&\pbf_{\minrm}- \pbf+\delta_{\Mbf\!}(\xbf),\\
% &\qbf - \qbf_{\maxrm}+\delta_{\Mbf\!}(\xbf),\quad
% &&\qbf_{\minrm}-\qbf+\delta_{\Mbf\!}(\xbf),\\
% &|\vec{\sbf}|^2\!\msh-\! \fbf^2_{\maxrm}\! +\!
% \delta_{\Mbf\!}(\xbf)^2\! +\! 2 \delta_{\Mbf\!}(\xbf)|\vec{\sbf}|,\\
% &|\cev{\sbf}|^2\! -\! \fbf^2_{\maxrm}\! +\!
% \delta_{\Mbf\!}(\xbf)^2\! +\! 2 \delta_{\Mbf\!}(\xbf)|\cev{\sbf}|,
% \end{aligned}
% \end{equation}
%\vspace{1mm}
%\vspace{-5mm}
%\noindent 
%respectively. 
Additionally, for every $\xbf$ that satisfies the Generalized LICQ condition, define $\sigma(\xbf)$ as the minimum singular value of the matrix $\Jbf_{\mathcal{B}^{\downarrow}_1,\mathcal{B}^{\uparrow}_1,
\mathcal{B}^{\downarrow}_2,\mathcal{B}^{\uparrow}_2, 
\mathcal{B}^{\downarrow}_3,\mathcal{B}^{\uparrow}_3,
\cev{\mathcal{B}}_4,\vec{\mathcal{B}}_4}(\xbf)$.
%\vspace{-2mm}
\end{definition}

%%%%%%%%%%%%%%%%%%%%%%%%%%%%%%%%%%%%%%%%%%%%%%%%%%%%%
\section{Convexification of the OPF Problem}
The proposed relaxation of the OPF problem involves three steps that are detailed in this section.

\vspace{-1mm}
\subsection{Lifting}
The nonlinear constraints \cref{eq:OPF_cons_1,eq:OPF_cons_2,eq:OPF_cons_3,eq:OPF_cons_4} and \cref{eq:OPF_cons_7,eq:OPF_cons_8}, as well as the objective function \cref{eq:OPF_obj} can be cast linearly by lifting the problem to a higher dimensional space. To this end, define the auxiliary variables $\obf,\rbf\in\Rbb^{|\Gcal|}$ and $\vec\fbf,\cev\fbf\in\Rbb^{|\Ecal|}$ accounting for $\pbf^{2}$, $\qbf^{2}$, $|\vec{\sbf}|^{2}$, and $|\cev{\sbf}|^{2}$, respectively. 
Moreover, define the auxiliary matrix variable $\Wbf\in\Hbb_{|\Vcal|}$, accounting for $\vbf\vbf^{\ast}$.
Observe that the constraints \cref{eq:OPF_cons_1,eq:OPF_cons_2,eq:OPF_cons_3,eq:OPF_cons_4} can be cast linearly with respect to $\Wbf\in\Hbb_{|\Vcal|}$. To preserve the relation between the original and lifted formulations, the following additional constraint shall be imposed:

\vspace{-6mm}
\begin{align}\label{eq:ncc}
\Wbf=\vbf\vbf^{\!\ast}\!.
\end{align}

\vspace{-1mm}
\noindent
The non-convexity of the lifted formulation is captured by the above constraint, which is addressed next.

\subsection{Convex Relaxation}
In order to make the OPF problem computationally tractable, it is common practice to relax the non-convex constraint 
\eqref{eq:ncc} to 

\vspace{-6mm}
\begin{align}
\Wbf-\vbf\vbf^{\ast}\in\Ccal,\label{cr}
\end{align}

\vspace{-1mm}
\noindent
where $\Ccal$ is a proper convex cone. In what follows, we discuss two commonly-used conic relaxations, as well as a novel relaxation which transforms the OPF problem \cref{eq:OPF_obj,eq:OPF_cons_1,eq:OPF_cons_2,eq:OPF_cons_3,eq:OPF_cons_4,eq:OPF_cons_5,eq:OPF_cons_6,eq:OPF_cons_7,eq:OPF_cons_8} into a convex quadratically-constrained quadratic program. 

\vspace{0mm}
% Convex Relaxation=======================================
\subsubsection{SDP Relaxation}
To derive a semidefinite programming (SDP) relaxation for the OPF problem \cref{eq:OPF_obj,eq:OPF_cons_1,eq:OPF_cons_2,eq:OPF_cons_3,eq:OPF_cons_4,eq:OPF_cons_5,eq:OPF_cons_6,eq:OPF_cons_7,eq:OPF_cons_8}, we can use the cone of $|\Vcal|\!\times\! |\Vcal|$ Hermitian positive semidefinite matrices:  
\begin{equation}\label{eq:SDP_Set}\notag
\Ccal_{1}\triangleq\big\{ \Hbf\in\Hbb_{|\Vcal|} \;\big|\; \Hbf\succeq 0 \big\}.
\end{equation}
Unlike the original non-convex problem \cref{eq:OPF_obj,eq:OPF_cons_1,eq:OPF_cons_2,eq:OPF_cons_3,eq:OPF_cons_4,eq:OPF_cons_5,eq:OPF_cons_6,eq:OPF_cons_7,eq:OPF_cons_8}, SDP relaxation is convex and is proven to result in a globally optimal solution for several benchmark cases of OPF \cite{lavaei2012zero}. 
% DSDP %%%%%%%%%%%%%%%%%%%%%%%%%%%%%%%%%%%%%%%%%%%%%%%
Despite the advantages of SDP relaxation, imposing a high-dimensional conic constraint can be computationally challenging. For sparse QCQP problems, the complexity of solving SDP relaxation can be alleviated through a graph-theoretic analysis, namely tree decomposition \cite{molzahn2013implementation,andersen2014reduced,bose2015equivalent,madani2016promises,guo2017case,zhang2017distributed}. Using a simple greedy algorithm \cite{madani2016promises}, $\Vcal$ can be decomposed into several overlapping subsets $\mathcal{A}_{1}, \mathcal{A}_{2}, \dots, \mathcal{A}_{D} \subseteq \mathcal{V}$, and then the relaxation is formulated in terms of the reduced cone:
\begin{equation}\notag
\Ccal_1^{\mathrm{d}}\triangleq\big\{\Hbf\in\Hbb_{|\Vcal|} \;\big|\; \Hbf\{\mathcal{A}_{k}, \mathcal{A}_{k}\}\succeq 0, \forall k \in \{1, 2, \cdots, D\}\big\},
\end{equation} 
where for each $k$, $\Hbf\{\mathcal{A}_{k}, \mathcal{A}_{k}\}$ represents the $|\mathcal{A}_{k}|\times|\mathcal{A}_{k}|$ principal sub-matrix of $\Hbf$ whose rows and columns are chosen from $\mathcal{A}_{k}$. The above decomposition leads to an equivalent but more tractable formulation of SDP relaxation. Nevertheless, solving large-scale instances of OPF on real-world systems can be still computationally challenging.
% SOCP %%%%%%%%%%%%%%%%%%%%%%%%%%%%%%%%%%%%%%%%%%%%%%%%%

\subsubsection{SOCP Relaxation}
A computationally cheaper alternative to SDP relaxation is the second-order cone programming (SOCP) relaxation, which is formulated using the cone
\begin{equation}\label{eq:SOCP_Set}\notag
\begin{aligned}
\Ccal_{2}\triangleq\big\{ \Hbf\in\Hbb_{|\Vcal|} \;\big|\; H_{ii}\geq\! 0,\;H_{ii} H_{jj}\geq|H_{ij}|^2,\;\forall (i,\!j)\!\in\!\Ecal \big\}.
\end{aligned}
\end{equation}
Incorporating $\Ccal_{2}$ into the constraint \eqref{cr} leads to the SOCP relaxation of OPF. Note that although the SDP relaxation is generally tighter, the SOCP relaxation is far more scalable.
% Para %%%%%%%%%%%%%%%%%%%%%%%%%%%%%%%%%%%%%%%%%%%%%%%%%%%%

\subsubsection{Parabolic Relaxation}
In order to avoid conic constraints, in this paper, we propose a computationally efficient method, regarded as the {\it parabolic relaxation}, which transforms an arbitrary non-convex QCQP into a convex QCQP. The proposed method requires far less computational effort and can serve as an alternative to the common practice SDP and SOCP relaxations for solving large-scale OPF problems. To derive the parabolic relaxation, define: %To derive the parabolic relaxation, one needs to incorporate the following convex set into constraint \eqref{cr}:
\begin{equation}\label{eq:Para_Set}\notag
\begin{aligned}
\Ccal_{3}\triangleq \big\{ \Hbf\in\Hbb_{|\Vcal|} \;\big|\; & H_{ii}\!\geq\! 0,H_{ii}\!+\!H_{jj}\!\geq\! 2\left|\mathrm{real}\{H_{ij}\}\right|\!, \\ &H_{ii}\!+\!H_{jj}\!\geq\! 2\left|\mathrm{imag}\{H_{ij}\}\right|\!,\forall (i,\!j)\!\in\!\Ecal \big\}.
\end{aligned}
\end{equation}
%\begin{remark}
%\vspace{1mm}
If $\Ccal_3$ is used, the constraint \eqref{cr} transforms to the following convex quadratic inequalities
\begin{subequations}
\begin{align}
\hspace{-0.3cm}|v_{i}-\phantom{i}v_{j}|^2 &\leq W_{ii}\!+\!W_{jj} - \phantom{\irm}(W_{ij}\!+\!W_{ji}) &&\hspace{-0.1cm} \forall (i,\!j)\!\in\!\Ecal \\
\hspace{-0.3cm}|v_{i}+\phantom{i}v_{j}|^2 &\leq W_{ii}\!+\!W_{jj} + \phantom{\irm}(W_{ij}\!+\!W_{ji}) &&\hspace{-0.1cm} \forall (i,\!j)\!\in\!\Ecal  \\
\hspace{-0.3cm}|v_{i}-\irm v_{j}|^2 &\leq W_{ii}\!+\!W_{jj} + \irm(W_{ij}\!-\!W_{ji}) &&\hspace{-0.1cm} \forall (i,\!j)\!\in\!\Ecal  \\
\hspace{-0.3cm}|v_{i}+\irm v_{j}|^2 &\leq W_{ii}\!+\!W_{jj} - \irm(W_{ij}\!-\!W_{ji}) &&\hspace{-0.1cm} \forall (i,\!j)\!\in\!\Ecal\\
\hspace{-0.3cm}|v_{i}|^2 &\leq W_{ii} &&~~~\, \forall i\!\in\!\Vcal 
\end{align}
\end{subequations}
%\vspace{0.01mm}
%\end{remark}
and there is no need to impose conic constraints.

\begin{definition}
For every $k\in\{1,2,3\}$, define $\mathcal{D}_k$ as the dual cone of $\Ccal_{k}$. Observe that the cone of Hermitian positive semidefinite matrices is self-dual, i.e., $\mathcal{D}_1=\mathcal{C}_1$. Moreover, $\mathcal{D}_{2}$ and $\mathcal{D}_{3}$ are, respectively, the sets of $|\mathcal{V}|\times|\mathcal{V}|$ Hermitian scaled-diagonally-dominant (SDD) and diagonally-dominant matrices, defined as,

\begin{subequations}\small
\begin{align}
\mathcal{D}_{2} &\!=\! \Big\{\!\!\!\! \sum_{(i,j)\in\Ecal}[\ebf_{i},\ebf_{j}]\,\Hbf_{ij}\,[\ebf_{i},\ebf_{j}]^{\!\top}\Big| \Hbf_{ij}\!\in\!\Hbb_{\,2},\, \Hbf_{ij}\!\succeq\! 0,\, \forall (i,\!\,j)\!\in\!\Ecal \Big\},\nonumber\\
\mathcal{D}_{3} &\!=\! \Big\{ \Hbf\in\Hbb_{|\Vcal|}\,\Big|\,|H_{ii}| \geq\!\!\sum_{j\in\Vcal\setminus\{i\}}\!|H_{ij}|,\,\forall i\!\in\!\Vcal \Big\},\nonumber
\end{align}
\end{subequations}
where $\{\ebf_{i}\}_{i\in |\Vcal|}$ represents the standard basis for $\Rbb^{|\Vcal|}$. Moreover, the interior of $\mathcal{D}_k$ can be expressed as 
\begin{align}
\!\!\!\mathrm{int}\{\mathcal{D}_k\}\!=\!
\left\{
\mathbf{M}\in\mathbb{H}_{|\mathcal{V}|}\,|\,
\exists\,\varepsilon>0;\, \mathbf{M}-\varepsilon\mathbf{I}_{|\mathcal{V}|}
\!\in\!\mathcal{D}_k\right\},\!\!\!
\end{align}
for every $k\in\{1,2,3\}$.
%Moreover, define the distance functions $\gamma:\Hbb_{|\Vcal|}\times 2^{\Hbb_{|\Vcal|}}\rightarrow\Rbb$ and $\delta:\Hbb_{|\Vcal|}\times 2^{\Hbb_{|\Vcal|}}\rightarrow\Rbb$ as,
% \begin{subequations}
% \begin{align}
% \gamma(\Mbf,\mathcal{D}_k)
% &\triangleq\textnormal{
% inf}_{\mathbf{N}\in\mathbb{H}_{|\Vcal|}} \{\,\|\mathbf{N}\|_1   \,|\,\Mbf+\mathbf{N}\notin\mathcal{D}_k\,\},\\
% \delta(\Mbf,\mathcal{D}_k)
% &\triangleq\textnormal{
% inf}_{\mathbf{N}\in\mathbb{H}_{|\Vcal|}} \{\,\|\mathbf{N}\|_2   \,|\,\Mbf+\mathbf{N}\notin\mathcal{D}_k\,\}.
% \end{align}
% \end{subequations}
% \vspace{1mm}
\end{definition} 

%\vspace{-2mm}
%\noindent
In practice, the aforementioned convex relaxations are not necessarily exact, which means that solutions obtained by solving the relaxed problems may not be feasible for the OPF problem \cref{eq:OPF_obj,eq:OPF_cons_1,eq:OPF_cons_2,eq:OPF_cons_3,eq:OPF_cons_4,eq:OPF_cons_5,eq:OPF_cons_6,eq:OPF_cons_7,eq:OPF_cons_8}. Next, we show that it is possible to resolve this issue and obtain near-optimal feasible points for OPF by incorporating a penalty term into the objective function of SDP, SOCP, and parabolic relaxations. %However, convex relaxations offer lower bounds on the optimal cost of OPF, which can be used to evaluate quality of near-optimal feasible points.

%However, given the fact that the optimal cost of convex relaxations are lower bounds for the optimal cost of the original non-convex OPF, the above relaxations (whether exact or inexact) provide a measure for the quality of near-optimal feasible points. Next, we propose a novel methods to revise the objective of the relaxed problem.
% Penalization %%%%%%%%%%%%%%%%%%%%%%%%%%%%%%%%%%%%%%%%%%%%%
\subsection{Penalization}
%A solution obtained from convex relaxation may not be feasible for the original non-convex problem \cref{eq:OPF_obj,eq:OPF_cons_1,eq:OPF_cons_2,eq:OPF_cons_3,eq:OPF_cons_4,eq:OPF_cons_5,eq:OPF_cons_6,eq:OPF_cons_7,eq:OPF_cons_8}. 
To address the inexactness of convex relaxations, we revise  objective functions %in order to find feasible solutions. To this end, 
by adding linear penalty terms of the form
$\kappa(\Wbf,\obf,\rbf,\vec{\fbf},\cev{\fbf},\vbf,\pbf\!+\!\irm\qbf,\vec{\sbf},\cev{\sbf})$, using which the non-convex constraint \eqref{eq:ncc} is implicitly imposed. 
Given an initial guess $\mathbf{x}_{\sm{0}{5}}=(\vbf_{\sm{0}{5}},\pbf_{\sm{0}{5}}\!+\!\irm\qbf_{\sm{0}{5}},\vec{\sbf}_{\sm{0}{5}},\cev{\sbf}_{\sm{0}{5}})$ for the solution of the OPF problem \cref{eq:OPF_obj,eq:OPF_cons_1,eq:OPF_cons_2,eq:OPF_cons_3,eq:OPF_cons_4,eq:OPF_cons_5,eq:OPF_cons_6,eq:OPF_cons_7,eq:OPF_cons_8}, the following definition introduces a family of penalty terms that guarantee the exactness of relaxation if $\xbf_{0}$ is sufficiently close to the set $\mathcal{F}$.

%%%%%%%%%%%%%%%%%%%%%%%%%%%%%%%%%%%%%%%%%%%%%%%%%%%%%%%%%%%%%%%%%%%%%%%%
\begin{definition}
% \vspace{.1mm}
Given an arbitrary initial point
$\xbf_{\sm{0}{5}}=(\vbf_{\sm{0}{5}},\pbf_{\sm{0}{5}}\!+\!\irm\qbf_{\sm{0}{5}},\vec{\sbf}_{\sm{0}{5}},\cev{\sbf}_{\sm{0}{5}})
\in\Cbb^{|\Vcal|}\!\times\!\Cbb^{|\Gcal|}
\!\times\!\mathbb{C}^{|\Ecal|}\!\times\!\Cbb^{|\Ecal|}$,
the penalty function $\kappa_{\Mbf,\xbf_0}$ is defined as follows
\begin{align}\label{eq:penalization}\notag
%\hspace{-9mm}
&\kappa_{\Mbf,\xbf_0}(\Wbf,\obf,\rbf,\vec{\fbf},\cev{\fbf},\vbf,\pbf\!+\!\irm\qbf,\vec{\sbf},\cev{\sbf})\triangleq\nonumber\\
&~~~\,(\obf^{\!\top}\!\onebf -2\,\pbf_{\sm{0}{5}}^{\!\top}\pbf+\pbf_{\sm{0}{5}}^{\!\top}\pbf_{\sm{0}{5}})+(\rbf^{\top}\onebf -2\,\qbf_{\sm{\,0}{5}}^{\!\top}\qbf+\qbf_{\sm{\,0}{5}}^{\!\top}\qbf_{\sm{\,0}{5}})+\nonumber\\ 
&~~~\,(\vec{\fbf}^{\msh\;\top}\!\onebf\msh -\vec{\sbf}_{\sm{0}{5}}^{\ast}\vec{\sbf}\msh-\msh \vec{\sbf}^{\ast}\vec{\sbf}_{\sm{0}{5}}\msh+\msh \vec{\sbf}_{\sm{0}{5}}^{\ast}\vec{\sbf}_{\sm{0}{5}})\!+\!(\cev{\fbf}^{\,\top}\!\onebf\msh -\cev{\sbf}_{\sm{0}{5}}^{\ast}\cev{\sbf}\msh-\msh \cev{\sbf}^{\ast}\cev{\sbf}_{\sm{0}{5}}\msh+\msh\cev{\sbf}_{\sm{0}{5}}^{\ast}\cev{\sbf}_{\sm{0}{5}})+\nonumber\\
&~~~\,\tr\{\Wbf\Mbf\}-\vbf_{\sm{0}{5}}^{\ast}\Mbf\vbf -\vbf^{\ast}\Mbf\vbf_{\sm{0}{5}}+\vbf_{\sm{0}{5}}^{\ast}\Mbf\vbf_{\sm{0}{5}},
\end{align}
where $\mathbf{M}\in\mathbb{H}_{|\Vcal|}$ is regarded as the penalty matrix. %to be designed.
% \vspace{.1mm}
\end{definition}
%%%%%%%%%%%%%%%%%%%%%%%%%%%%%%%%%%%%%%%%%%%%%%%%%%%%%%%%%%%%%%%%%%%%%%%%

Given $k\in\{1,2,3\}$, a penalty matrix $\mathbf{M}$, and $\mu>0$, the \textit{penalized relaxation} problem equipped with the cone $\mathcal{C}_k$ and the penalty term $\mu\times\kappa_{\Mbf,\xbf_0}$ can be formulated as: 
%=====================================================================
\begin{subequations}\label{eq:OPF_lifted}
	\begin{align}
	&\hspace{-0.6cm}\underset{\begin{subarray}{c} 
\pbf,\qbf,\obf,\rbf\in\Rbb^{|\Gcal|}\\ 
\hspace{3.0mm}\vbf\in\Cbb^{|\msh\Vcal\msh|}, \Wbf\in\Hbb_{|\msh\Vcal\msh|} \\ \hspace{4.2mm}\vec{\sbf},\cev{\sbf}\in\Cbb^{|\msh\Ecal\msh|},\vec{\fbf},\cev{\fbf}\in\Rbb^{|\msh\Ecal\msh|}\end{subarray}
	}{\hspace{-0.1cm}\text{minimize}}
	&&\hspace{-0.6cm} \msh h_{\mathrm{L}\!}(\msh\obf,\msh\pbf\msh)\!+\!\mu\kappa_{\Mbf,\xbf_0}\msh(\Wbf\!,\msh\obf,\msh\rbf,\vec{\fbf}\msh,\msh\cev{\fbf}\msh,\msh\msh\vbf,\msh\pbf\!+\!\irm\qbf,\vec{\sbf},\cev{\sbf})\!\!\!\!
	\label{eq:OPF_lifted_obj}\\
	&\hspace{-0.35cm}\text{~~subject to}
	&&\hspace{-0.625cm}\dbf + \mathrm{diag}\{\Wbf\,\Ybf^{\ast}\}=\Cbf^{\top}(\pbf+\irm\qbf)\label{eq:OPF_lifted_cons_1}\\
	&&&\hspace{-0.625cm}\mathrm{diag}\{\vec{\Cbf}\;\!\Wbf\,\vec{\Ybf}^{\ast}\}={\vec{\sbf}}&&\hspace{-1.3 em}\label{eq:OPF_lifted_cons_2}\\
	&&&\hspace{-0.625cm}\mathrm{diag}\{\cev{\Cbf}\;\!\Wbf\,\cev{\Ybf}^{\ast}\}={\cev{\sbf}}&&\hspace{-1.3 em}\label{eq:OPF_lifted_cons_3}\\
	&&&\hspace{-0.625cm}\vbf_{\mathrm{min}}^{2}\!\leq\mathrm{diag}\{\Wbf\}\leq \vbf_{\mathrm{max}}^{2}\label{eq:OPF_lifted_cons_4}\\
	&&&\hspace{-0.625cm}{\pbf_{\mathrm{min}}} \leq  \pbf \leq {\pbf_{\mathrm{max}}}\label{eq:OPF_lifted_cons_5}\\ 
	&&&\hspace{-0.625cm}{\qbf_{\mathrm{min}}} \leq  \qbf \leq {\qbf_{\mathrm{max}}}\label{eq:OPF_lifted_cons_6}\\
	&&&\hspace{-0.625cm}|\vec{\sbf}|^2 \leq \vec{\fbf} \leq \fbf_{\mathrm{max}}^2\label{eq:OPF_lifted_cons_7}\\
	&&&\hspace{-0.625cm}|\cev{\sbf}|^2 \leq \cev{\fbf} \leq \fbf_{\mathrm{max}}^2\label{eq:OPF_lifted_cons_8}\\
	&&&\hspace{-0.625cm}\pbf^2 \leq \obf\label{eq:OPF_lifted_cons_10}\\
	&&&\hspace{-0.625cm}\Wbf - \vbf\;\!\vbf^{\ast}\!\in \Ccal_k\label{eq:OPF_lifted_cons_11}
	\end{align}
\end{subequations}
where $h_{\mathrm{L}}(\obf,\pbf)\!\!\triangleq\!\cbf_{0}^{\!\top}\onebf\!+\cbf_{1}^{\top}\!\pbf+\cbf_{2}^{\top}\obf$ is the lifted objective function. 
%Observe that problems \cref{eq:OPF_lifted_obj,eq:OPF_lifted_cons_1,eq:OPF_lifted_cons_2,eq:OPF_lifted_cons_3,eq:OPF_lifted_cons_4,eq:OPF_lifted_cons_5,eq:OPF_lifted_cons_6,eq:OPF_lifted_cons_7,eq:OPF_lifted_cons_8,eq:OPF_lifted_cons_10,eq:OPF_lifted_cons_11} and \cref{eq:OPF_obj,eq:OPF_cons_1,eq:OPF_cons_2,eq:OPF_cons_3,eq:OPF_cons_4,eq:OPF_cons_5,eq:OPF_cons_6,eq:OPF_cons_7,eq:OPF_cons_8} are equivalent if $\mu=0$ and $\Ccal=\{\mathbf{0}\}$. 
The penalization is said to be tight if the problem \cref{eq:OPF_lifted_obj,eq:OPF_lifted_cons_1,eq:OPF_lifted_cons_2,eq:OPF_lifted_cons_3,eq:OPF_lifted_cons_4,eq:OPF_lifted_cons_5,eq:OPF_lifted_cons_6,eq:OPF_lifted_cons_7,eq:OPF_lifted_cons_8,eq:OPF_lifted_cons_10,eq:OPF_lifted_cons_11} possesses a unique solution that satisfies the equation \!\eqref{eq:ncc}. The tightness of penalization guarantees the recovery of a feasible point for the OPF problem \cref{eq:OPF_obj,eq:OPF_cons_1,eq:OPF_cons_2,eq:OPF_cons_3,eq:OPF_cons_4,eq:OPF_cons_5,eq:OPF_cons_6,eq:OPF_cons_7,eq:OPF_cons_8}.

\section{Theoretical Results}
%%%%%%%%%%%%%%%%%%%%%%%%%%%%%%%%%%%%
%%%%%%%%%%%%%%%%%%%%%%%%%%%%%%%%%%%%

%The following theorem shows that if the penalized convex relaxation \cref{eq:OPF_lifted_obj,eq:OPF_lifted_cons_1,eq:OPF_lifted_cons_2,eq:OPF_lifted_cons_3,eq:OPF_lifted_cons_4,eq:OPF_lifted_cons_5,eq:OPF_lifted_cons_6,eq:OPF_lifted_cons_7,eq:OPF_lifted_cons_8,eq:OPF_lifted_cons_10,eq:OPF_lifted_cons_11} is built based on a feasible point,

It is shown in \cite{hauswirth2018generic} that the LICQ condition holds generically for OPF. According to the next theorem, if $\xbf_0$ is a feasible point for the OPF problem \cref{eq:OPF_obj,eq:OPF_cons_1,eq:OPF_cons_2,eq:OPF_cons_3,eq:OPF_cons_4,eq:OPF_cons_5,eq:OPF_cons_6,eq:OPF_cons_7,eq:OPF_cons_8} that satisfies the LICQ condition, then the penalized convex relaxation problem \cref{eq:OPF_lifted_obj,eq:OPF_lifted_cons_1,eq:OPF_lifted_cons_2,eq:OPF_lifted_cons_3,eq:OPF_lifted_cons_4,eq:OPF_lifted_cons_5,eq:OPF_lifted_cons_6,eq:OPF_lifted_cons_7,eq:OPF_lifted_cons_8,eq:OPF_lifted_cons_10,eq:OPF_lifted_cons_11} preserves the feasibility of $\xbf_0$ for appropriate choices of the penalty matrix $\Mbf$ and $\mu$.
\begin{theorem}\label{thm:feas}
\vspace{1mm}
Let $\mathbf{x}_{\sm{0}{5}}=(\vbf_{\sm{0}{5}},\pbf_{\sm{0}{5}}+\irm\qbf_{\sm{0}{5}},\vec{\sbf}_{\sm{0}{5}},\cev{\sbf}_{\sm{0}{5}})\in\mathcal{F}$ be a feasible point for the OPF problem \cref{eq:OPF_obj,eq:OPF_cons_1,eq:OPF_cons_2,eq:OPF_cons_3,eq:OPF_cons_4,eq:OPF_cons_5,eq:OPF_cons_6,eq:OPF_cons_7,eq:OPF_cons_8}, which satisfies the LICQ condition. Assume that $\mathbf{M}\in\mathrm{int}\{\mathcal{D}_k\}$, where $k\in\{1,2,3\}$. If $\mu$ is sufficiently large, then the penalized convex relaxation \cref{eq:OPF_lifted_obj,eq:OPF_lifted_cons_1,eq:OPF_lifted_cons_2,eq:OPF_lifted_cons_3,eq:OPF_lifted_cons_4,eq:OPF_lifted_cons_5,eq:OPF_lifted_cons_6,eq:OPF_lifted_cons_7,eq:OPF_lifted_cons_8,eq:OPF_lifted_cons_10,eq:OPF_lifted_cons_11}, equipped with the cone $\mathcal{C}_k$ and the penalty term $\mu\times\kappa_{\Mbf,\xbf_0}$ has a unique solution
\begin{align}
(\Wbf_{\sm{\mathrm{opt}}{6}},
\obf_{\sm{\mathrm{opt}}{6}},
\rbf_{\sm{\mathrm{opt}}{6}},
\vec{\fbf}_{\sm{\mathrm{opt}}{6}},
\cev{\fbf}_{\sm{\mathrm{opt}}{6}},
\vbf_{\sm{\mathrm{opt}}{6}},
\pbf_{\sm{\mathrm{opt}}{6}},
\qbf_{\sm{\mathrm{opt}}{6}},
\vec{\sbf}_{\sm{\mathrm{opt}}{6}},
\cev{\sbf}_{\sm{\mathrm{opt}}{6}}),\nonumber
\end{align}
such that 
$\mathbf{x}_{\sm{\mathrm{opt}}{6}}\triangleq(\vbf_{\sm{\mathrm{opt}}{6}},\pbf_{\sm{\mathrm{opt}}{6}}\!+\!\irm\qbf_{\sm{\mathrm{opt}}{6}},\vec{\sbf}_{\sm{\mathrm{opt}}{6}},\cev{\sbf}_{\sm{\mathrm{opt}}{6}})$
is feasible for the original OPF problem \cref{eq:OPF_obj,eq:OPF_cons_1,eq:OPF_cons_2,eq:OPF_cons_3,eq:OPF_cons_4,eq:OPF_cons_5,eq:OPF_cons_6,eq:OPF_cons_7,eq:OPF_cons_8} and $h(\pbf_{\sm{\mathrm{opt}}{6}})\leq h(\pbf_{\sm{0}{5}})$.
\vspace{1mm}
\end{theorem}
%%%%%%%%%%%%%%%%%%%%%
\begin{proof}
The theorem is proven in \cite{madani2018_1} for the more general case of optimization problems with bilinear matrix inequality (BMI) constraints. The proof for penalized SDP and SOCP relaxations of QCQPs is given in \cite{QCQP_conic} %A similar proof can be used here since the OPF problem is a  QCQP, which is a special case of BMI optimization.
\vspace{1mm}
\end{proof}
%%%%%%%%%%%%%%%%%%%%%%%%%%%%%%%%%%%%%%%%%%%%%%%%%%%%%%%%%%%%%%%%%%%%%%%%

%In many cases, obtaining a feasible point is computationally demanding and hard. To address this issue, we propose some conditions under which the penalized relaxation problem
%\cref{eq:OPF_lifted_obj,eq:OPF_lifted_cons_1,eq:OPF_lifted_cons_2,eq:OPF_lifted_cons_3,eq:OPF_lifted_cons_4,eq:OPF_lifted_cons_5,eq:OPF_lifted_cons_6,eq:OPF_lifted_cons_7,eq:OPF_lifted_cons_8,eq:OPF_lifted_cons_10,eq:OPF_lifted_cons_11} is guaranteed to recover a feasible solution from an entirely infeasible one.

Obtaining a feasible point for OPF may not be straightforward. The next theorem is concerned with the case where the initial point $\xbf_0$ is not feasible.

%%%%%%%%%%%%%%%%%%%%%%%%%%%%%%%%%%%%%%%%%%%%%%%%%%%%%%%%%%%%%%%%%%%%%%%%
% \begin{definition}\label{def:const}
% \vspace{1mm}
% For an arbitrary Hermitian positive-definite matrix $\Mbf$, we define $\omega_k$, for cone $\Ccal_k$, $k\in\{1,2,3\}$, as, 
% \begin{subequations}
% \begin{align}
% \omega_1(\Mbf)&\triangleq
% \frac{{\lambda}_{\max}(\Mbf)^{-\frac{1}{2}}}
% {4{\lambda}_{\min}(\Mbf)^{-1}}\\
% \omega_2(\Mbf)&\triangleq
% \frac{{\lambda}_{\max}(\Mbf)^{-\frac{1}{2}}}{2({\lambda}_{\min}(\Mbf)^{-1}+\delta(\Mbf,\mathcal{D}_2)^{-1})}\\
% \omega_3(\Mbf)&\triangleq
% \frac{{\lambda}_{\max}(\Mbf)^{-\frac{1}{2}}}{2({\lambda}_{\min}(\Mbf)^{-1}+|\mathcal{V}|^{\frac{1}{2}}\gamma(\Mbf,\mathcal{D}_3)^{-1})}
%     \end{align}
% \end{subequations}
% where ${\lambda}_{\min}(\Mbf)$ and ${\lambda}_{\max}(\Mbf)$, respectively, return the smallest and largest eigenvalue of $\Mbf$. \end{definition}
%%%%%%%%%%%%%%%%%%%%%%%%%%%%%%%%%%%%%%%%%%%%%%%%%%%%%%%%%%%%%%%%%%%%%%%%

%The next theorem presents an upper-bound and requirements to guarantee the exactness of the penalized relaxation.
%%%%%%%%%%%%%%%%%%%%%%%%%%%%%%%%%%%%%%%%%%%%%%%%%%%%%%%%%%%%%%%%%%%%%%%%
\begin{theorem}\label{thm:nfeas}
\vspace{1mm}
Consider an arbitrary point $\xbf_{\sm{0}{5}}=(\vbf_{\sm{0}{5}},\pbf_{\sm{0}{5}}\!+\!\irm\qbf_{\sm{0}{5}},\vec{\sbf}_{\sm{0}{5}},\cev{\sbf}_{\sm{0}{5}})\in\Cbb^{|\Vcal|}\!\times\!\Cbb^{|\Gcal|}
\!\times\!\Cbb^{|\Ecal|}\!\times\!\Cbb^{|\Ecal|}$, which satisfies the Generalized LICQ condition. Assume that $\mathbf{M}-\mathbf{I}_{|\mathcal{V}|}\!\in\!\mathrm{int}\{\mathcal{D}_k\}$ and
\begin{align}
\mathbf{M}\preceq
\left(\frac{\sigma(\xbf_0)}{4 \delta_{\Mbf}(\xbf_0)P}\right) \mathbf{I}_{|\mathcal{V}|},
\end{align}
where $k\in\{1,2,3\}$, and $P$, $\delta_{\Mbf}(\xbf_0)$ and $\sigma(\xbf_0)$ are given by Definitions \ref{def:sens}, \ref{def:feasdist} and \ref{def:GLICQ}, respectively. If $\mu$ is sufficiently large, then the penalized relaxation problem \cref{eq:OPF_lifted_obj,eq:OPF_lifted_cons_1,eq:OPF_lifted_cons_2,eq:OPF_lifted_cons_3,eq:OPF_lifted_cons_4,eq:OPF_lifted_cons_5,eq:OPF_lifted_cons_6,eq:OPF_lifted_cons_7,eq:OPF_lifted_cons_8,eq:OPF_lifted_cons_10,eq:OPF_lifted_cons_11}, equipped with the cone $\mathcal{C}_k$ and the penalty term $\mu\times\kappa_{\Mbf,\xbf_0}$ has a unique solution
\begin{align}
(\Wbf_{\sm{\mathrm{opt}}{6}},
\obf_{\sm{\mathrm{opt}}{6}},
\rbf_{\sm{\mathrm{opt}}{6}},
\vec{\fbf}_{\sm{\mathrm{opt}}{6}},
\cev{\fbf}_{\sm{\mathrm{opt}}{6}},
\vbf_{\sm{\mathrm{opt}}{6}},
\pbf_{\sm{\mathrm{opt}}{6}},
\qbf_{\sm{\mathrm{opt}}{6}},
\vec{\sbf}_{\sm{\mathrm{opt}}{6}},
\cev{\sbf}_{\sm{\mathrm{opt}}{6}}),\nonumber
\end{align}
such that 
$\mathbf{x}_{\sm{\mathrm{opt}}{6}}\triangleq(\vbf_{\sm{\mathrm{opt}}{6}},\pbf_{\sm{\mathrm{opt}}{6}}\!+\!\irm\qbf_{\sm{\mathrm{opt}}{6}},\vec{\sbf}_{\sm{\mathrm{opt}}{6}},\cev{\sbf}_{\sm{\mathrm{opt}}{6}})$
is feasible for the original OPF problem \cref{eq:OPF_obj,eq:OPF_cons_1,eq:OPF_cons_2,eq:OPF_cons_3,eq:OPF_cons_4,eq:OPF_cons_5,eq:OPF_cons_6,eq:OPF_cons_7,eq:OPF_cons_8}.
\vspace{1mm}
\end{theorem}
%%%%%%%%%%%%%%%%%%%%%%%%%%%%%%%%%%%%%%%%%%%%%%%%%%%%%%%%%%%%%%%%%%%%%%%%
\begin{proof}
The proof can be found in {\cite{madani2018_1}}.
\vspace{1mm}
\end{proof}
%%%%%%%%%%%%%%%%%%%%%%%%%%%%%%%%%%%%%%%%%%%%%%%%%%%%%%%%%%%%%%%%%%%%%%%%

%%%%%%%%%%%%%%%%%%%%%%%%%%%%%%%%%%%%%%%%%%%%%%%%%%%%%%%%%%%%%%%%%%%%%%%%

\subsection{Choice of the Penalty Matrix}
Motivated by the previous literature \cite{madani2016promises,madani2015convex}, we propose to choose $\Mbf$ such that the term $\tr\{\Wbf\Mbf\}$ in the penalty function reduces the apparent power loss over the series element of every line in the network. 
According to \cref{eq:series_fpower,eq:series_tpower}, the apparent power loss over the series admittance $y_{s}$ can be expressed in terms of $\vbf_{l}$ and the admittance matrices $\vec{\Ybf}_{\mathrm{p};\,l}$, 
$\vec{\Ybf}_{\mathrm{q};\,l}$,
$\cev{\Ybf}_{\mathrm{p};\,l}$, and $\cev{\Ybf}_{\mathrm{q};\,l}$. Hence, in order to penalize the apparent power loss over all lines of the network, we choose the matrix $\Mbf$ as, 
\begin{equation}\label{eqM}
\begin{aligned}
\Mbf =\!\!\!\sum_{(i,j)\in\Ecal} [\ebf_{i}, \ebf_{j}](\Mbf_{ij}\!+\!\alpha\Ibf_{2})[\ebf_{i}, \ebf_{j}]^{\!\top}\!,
\end{aligned}
\end{equation}
where $\ebf_{1},\ldots,\ebf_{n}$ denote the standard basis for $\mathbb{R}^n$, $\alpha$ is a positive constant and each $\Mbf_{ij}$ is a $2\times 2$ positive semidefinite matrix defined as,
\begin{equation}\label{eqMij}
\begin{aligned}
\Mbf_{ij} = 
\zeta_{ij}(\vec{\Ybf}_{\mathrm{q};\,l}+
\cev{\Ybf}_{\mathrm{q};\,l})
+\frac{\eta}{1\!-\!\eta}(\vec{\Ybf}_{\mathrm{p};\,l}+\cev{\Ybf}_{\mathrm{p};\,l}).
\end{aligned}
\end{equation}
The parameter $\eta>0$ sets the trade-off between active and reactive loss minimization and $\zeta_{ij}\in\{-1,+1\}$ is determined based on the inductive or capacitive behavior of the line $l\!\in\!\Ecal$. More precisely, we set $\zeta_{ij}=1$ if the series admittance $y_{\sm{\mathrm{srs},\,l}{6}}$ is inductive (i.e., $b_{\sm{\mathrm{srs},\,l}{6}}\leq 0$), and $\zeta_{ij}=-1$, otherwise. %In the case where the system under study has capacitive series admittance, the rank-1 matrix $\vec{\Ybf}_{\mathrm{q};\,l}+\cev{\Ybf}_{\mathrm{q};\,l}$ is negative semidefinite which may result in non-convex penalty function $\kappa_{\Mbf,\xbf_0}$. In such cases, we consider $\zeta_{ij}\!=\!-1$ to render $\Mbf_{ij}$ positive semidefinite. 
Observe that if $\alpha$ is sufficiently large, then $\Mbf$ belongs to the relative interior of the dual cones $\mathcal{D}_1$, $\mathcal{D}_2$, and $\mathcal{D}_3$.
%=====================Sequential=================================
\subsection{Sequential Convex Relaxation}
The penalized convex relaxation 
\cref{eq:OPF_lifted_obj,eq:OPF_lifted_cons_1,eq:OPF_lifted_cons_2,eq:OPF_lifted_cons_3,eq:OPF_lifted_cons_4,eq:OPF_lifted_cons_5,eq:OPF_lifted_cons_6,eq:OPF_lifted_cons_7,eq:OPF_lifted_cons_8,eq:OPF_lifted_cons_10,eq:OPF_lifted_cons_11} can be solved sequentially to find near-globally optimal solutions for OPF. The details of this sequential procedure are delineated  by Algorithm \ref{al:alg_1}. According to Theorem \ref{thm:feas}, once a feasible point for the OPF problem %$\xbf\in\mathcal{F}$ 
\cref{eq:OPF_obj,eq:OPF_cons_1,eq:OPF_cons_2,eq:OPF_cons_3,eq:OPF_cons_4,eq:OPF_cons_5,eq:OPF_cons_6,eq:OPF_cons_7,eq:OPF_cons_8}
is obtained, feasibility is preserved, and the objective value improves in each round.
%Penalizing the apparent power loss over lines of the network often yields a near-globally optimal solution for OPF. However, the solution is not guaranteed to be feasible for the original problem \cref{eq:OPF_obj,eq:OPF_cons_1,eq:OPF_cons_2,eq:OPF_cons_3,eq:OPF_cons_4,eq:OPF_cons_5,eq:OPF_cons_6,eq:OPF_cons_7,eq:OPF_cons_8}. Indeed, as shown in Theorem \eqref{thm:nfeas}, the exactness of the penalized relaxation heavily relies on the distance of the arbitrary point $(\vbf_{\sm{0}{5}},\pbf_{\sm{0}{5}}+\irm\qbf_{\sm{0}{5}},\vec\sbf_{\sm{0}{5}},\cev\sbf_{\sm{0}{5}})$ from the feasible set of the OPF problem. To tackle this issue, we employ the penalized relaxation in a sequential manner to develop an algorithm which constantly reduces the level of inexactness. Starting from an arbitrary solution, the algorithm first finds a feasible solution for \cref{eq:OPF_obj,eq:OPF_cons_1,eq:OPF_cons_2,eq:OPF_cons_3,eq:OPF_cons_4,eq:OPF_cons_5,eq:OPF_cons_6,eq:OPF_cons_7,eq:OPF_cons_8} in a few number of rounds. Then, it preserves the exactness and sequentially finds a point with smaller objective value. 
% \vspace{5cm}
\begin{algorithm}
 \caption{Sequential Penalized Convex Relaxation.}\label{alg:1}
 \vspace{0.2mm}
 \begin{algorithmic}[1]
%  \vspace{3mm}
 %------Input----------
 \Require {$k\in\{1,2,3\}$,\; $\Mbf\in\mathrm{int}\{\mathcal{D}_k\}$,\; $\mu>0$,\; and
\Statex \hspace{5.4mm}$\mathbf{x}_{\sm{0}{5}}\!=\!(\vbf_{\sm{0}{5}},\pbf_{\sm{0}{5}}\!+\!\irm\qbf_{\sm{0}{5}},\vec\sbf_{\sm{0}{5}},\cev\sbf_{\sm{0}{5}})\in\Cbb^{|\Vcal|}\!\times\!\Cbb^{|\Gcal|}
\!\times\!\mathbb{C}^{|\Ecal|}\!\times\!\Cbb^{|\Ecal|}$}
 %----------------------
 \Repeat
% \State Solve the problem \cref{eq:OPF_lifted_obj,eq:OPF_lifted_cons_1,eq:OPF_lifted_cons_2,eq:OPF_lifted_cons_3,eq:OPF_lifted_cons_4,eq:OPF_lifted_cons_5,eq:OPF_lifted_cons_6,eq:OPF_lifted_cons_7,eq:OPF_lifted_cons_8,eq:OPF_lifted_cons_10,eq:OPF_lifted_cons_11}, equipped with the 
% \Statex \hspace{5.4mm}cone $\mathcal{C}_k$ and the penalty term $\mu\times\kappa_{\Mbf,\xbf_0}$ to obtain:
% \Statex \quad\,
% $\xbf_{\sm{\mathrm{opt}}{6}}=(\vbf_{\sm{\mathrm{opt}}{6}},
%  \pbf_{\sm{\mathrm{opt}}{6}}+ \irm\qbf_{\sm{\mathrm{opt}}{6}}, \vec\sbf_{\sm{\mathrm{opt}}{6}},
%  \cev\sbf_{\sm{\mathrm{opt}}{6}})$
\State Obtain
$\xbf_{\sm{\mathrm{opt}}{6}}=(\vbf_{\sm{\mathrm{opt}}{6}},
 \pbf_{\sm{\mathrm{opt}}{6}}+ \irm\qbf_{\sm{\mathrm{opt}}{6}}, \vec\sbf_{\sm{\mathrm{opt}}{6}},
 \cev\sbf_{\sm{\mathrm{opt}}{6}})$ 
by solv- 
\Statex \hspace{5.4mm}ing the optimization \cref{eq:OPF_lifted_obj,eq:OPF_lifted_cons_1,eq:OPF_lifted_cons_2,eq:OPF_lifted_cons_3,eq:OPF_lifted_cons_4,eq:OPF_lifted_cons_5,eq:OPF_lifted_cons_6,eq:OPF_lifted_cons_7,eq:OPF_lifted_cons_8,eq:OPF_lifted_cons_10,eq:OPF_lifted_cons_11}, equipped with the
\Statex \hspace{5.4mm}cone $\mathcal{C}_k$ and the penalty term $\mu\times\kappa_{\Mbf,\xbf_0}$.
 \State $\xbf_{\sm{0}{5}}\gets\xbf_{\sm{\mathrm{opt}}{6}}$.
 \Until {stopping criteria is met.}
 \Ensure {$\mathbf{x}_{\sm{0}{5}}=(\vbf_{\sm{0}{5}},\pbf_{\sm{0}{5}}\!+\!\irm\qbf_{\sm{0}{5}},\vec\sbf_{\sm{0}{5}},\cev\sbf_{\sm{0}{5}})$}
 %----------------
 \end{algorithmic}\label{al:alg_1}
 \end{algorithm}
 \vspace{-3.0mm}
 %\end{figure}

%========================= Experiment
\section{Experimental Results}
In this section, we detail our experiments for verifying the efficacy of the proposed methods. We consider the IEEE and European test cases from \matpower \cite{zimmerman2011matpower}, modified-IEEE test cases from \cite{bukhsh2013local}, and test cases from the NESTA v0.7.0 archive \cite{coffrin2014nesta}. All numerical experiments are performed in MATLAB using a 64-bit computer with an Intel 3.0 GHz, 12-core CPU, and 256 GB RAM. The CVX package version 3.0, SDPT3 version 4.0, and MOSEK version 8.0 are used for convex optimization.

%%%%%%%%%%%%%%%%%%%%%%%%%%%%%%%%%%%
% %%%%%%%%%%%%%%%%%%%%%%%%%%%%%%%%%%%%%%%%%%%%%%%%%%%%%%%%%%%%%%%%%%%%%%%%%
\begin{table}
\vspace{2mm}
\caption{The lower bounds and run times of parabolic relaxation compared to the SDP and SOCP relaxations.}
\scriptsize
	\centering
	\begin{tabular}{|c| @{\,}|@{\!\!\:\,}c@{\,}|@{\,}c@{\,}|@{\,}| @{\,}c@{\,}|@{\,}c@{\,}|@{\,}| @{\,}c@{\,}|@{\,}c@{\,}|}
		\hline
        \multirow{2}{*}{Test Cases} & \multicolumn{2}{c|@{\,}|@{\,}}{SDP}& \multicolumn{2}{c|@{\,}|@{\,}}{SOCP}& \multicolumn{2}{c@{\,}|}{Parabolic}\\
\cline{2-7}
& $\sm{\mathrm{LB}}{6}$ & $\sm{\mathrm{time}}{6}$ & $\sm{\mathrm{LB}}{6}$ & \sm{\mathrm{time}}{6} & \sm{\mathrm{LB}}{6} & $\sm{\mathrm{time}}{6}$ \\
		\hline
		\hline
         9  	 & 5296.69 & 1.14 & 5296.67 & 0.52  & 5216.03 & 0.66 \\ \hline
         14 	 & 8081.53 & 0.81 & 8075.12 & 0.51 & 7642.59 & 0.59 \\ \hline
         30 	 & 576.89  & 1.12 & 573.58 & 0.64 & 565.21 & 0.65 \\ \hline
         39  	 & 41862.08 & 0.96 & 41854.65 & 0.54 & 41216.34 & 0.79 \\ \hline
         57      & 41737.79 & 1.98 & 41711.01 & 0.92 & 41006.74 & 0.90\\ \hline
         118  	 & 129654.63 & 2.53 & 129341.96 & 1.68 & 125947.88 & 1.14\\ \hline
         300  	 & 719711.69 & 6.56 & 718654.29 & 5.83 & 705814.84 & 2.64\\ \hline
%%%%%%%%%%%%%%%%%%%%%%%%%%%%%%% pegase %%%%%%%%%%%%%%%%%%%%%%%%%%%%%%%%%%%
		 89pegase      & 5819.67  & 5.69 & 5810.17 & 2.91 & 5730.95 & 1.59\\ \hline
		 1354pegase    & 74062.53 &577.57 & 74012.39 &14.98 & 73027.96 &9.98\\ \hline
         2869pegase    & 133988.93 &4267.37 & 133880.03 &32.33 & 132381.10 &24.91\\ \hline
	\end{tabular} 
	\label{tab:relaxation_quality_tb}
    \vspace{-2.5mm}
\end{table}
%%%%%%%%%%%%%%%%%%%%%%%%%%%%%%%%%%%%%%%%%%%%%%%%%%%%%%%%%%%%%%%%%%%%%%%%%
% \begin{table}
% \scriptsize
% 	\centering
% 	\begin{tabular}{|c||    @{\;}c@{\;}|    @{\;}c@{\;}|    @{\;}c@{\;}|}
% 		\hline
% 		Test Cases & SDP & SOCP &  Parabolic \\
% 		\hline
% 		\hline
%          6ww 	        & 3143.97  & 3124.28  & 3046.41 \\ \hline
%          9  		    & 5296.69  & 5296.67  & 5216.03 \\ \hline
%          14 		    & 8081.53  & 8075.12  & 7642.59 \\ \hline
%          30 		    & 576.89   & 573.58   & 565.21  \\ \hline
%          ieee30        & 8906.14  & 8902.40  & 8343.40  \\ \hline
%          39  & 41862.08 & 41854.65 & 41216.34 \\ \hline
%          57             & 41737.79 & 41711.01 & 41006.74 \\ \hline
%          118  		    & 129654.63& 129341.96& 125947.88 \\ \hline
%          300  		    & 719711.69& 718654.29& 705814.84  \\ \hline
% %%%%%%%%%%%%%%%%%%%%%%%%%%%%%%% pegase %%%%%%%%%%%%%%%%%%%%%%%%%%%%%%%%%%%
% 		 89pegase      & 5819.67  & 5810.17  & 5730.95 \\ \hline
% 		 1354pegase    & 74062.53& 74012.39 & 73027.96 \\ \hline
%          2869pegase    & 133988.93  & 133880.03  & 132381.10 \\ \hline
%         %
% 	\end{tabular} 
% \caption{Solution quality of the Parabolic relaxation comparing to the quality of SDP and SOCP relaxations.}
% 	\label{tab:relaxation_quality_tb}

%     \vspace{-2.5mm}
% \end{table}

%%%%%%%%%%%%%%%%%%%%%%%%%%%%%%%%%%%
Table \ref{tab:relaxation_quality_tb} reports the optimal objective values for SDP, SOCP and parabolic relaxations %we report the lower bounds and running times obtained from the three relaxations (without penalty term) in Table \ref{tab:relaxation_quality_tb} 
for a number of IEEE and European benchmark systems. The lower bounds obtained from parabolic relaxation are close to the lower bounds offered by SDP and SOCP relaxations. Additionally, the running times for solving convex relaxations using MOSEK 8.0 are reported by the table. %requires lower computation time for large-scale systems. 
% In order to evaluate the solution quality of the parabolic relaxation, we compare the lower bound obtained by this relaxation to that of the SDP and SOCP relaxations. The results are reported in Table \ref{tab:relaxation_quality_tb} for the IEEE and European test cases. It can be observed that the lower bound obtained by parabolic relaxation is fairly close to the lower bounds provided by SDP and SOCP relaxations. Additionally, the computation time 
% \textcolor{blue}{Running time of solving the SDP, SOCP, and parabolic relaxations for case1354pegase system, using MOSEK, are 577.57s, 14.98s, and 9.98s, respectively.}
%%%%%%%%%%%%%%%%%%%%%%% Plot_01 
%%%%%%%%% ABSTRACT
\tikzset{
    master/.style={
        execute at end picture={
            \coordinate (lower right) at (current bounding box.south east);
            \coordinate (upper left) at (current bounding box.north west);
        }
    },
    slave/.style={
        execute at end picture={
            \pgfresetboundingbox
            \path (upper left) rectangle (lower right);
        }
    }
}
%%%%%%%%%%%%%%%%%%%%%%%%%%%%%%%%%%%%%%%%%%%%%%%%%%%%%%%%%%%%%%%%%%%%%%%%%
\begin{figure}
\vspace{0.08cm}
\captionsetup[subfigure]{position=b}
\centering
\begin{tikzpicture}[master,scale=0.85]
\centering
\begin{axis}[
		width = 0.52\textwidth,
		height = 0.29\textwidth,
        xmode= normal,
		ymode= log,	
    	ylabel= {$\tr\{\Wbf\!-\!\vbf\vbf^{\ast}\}$},
    	ylabel shift= -0.2cm,
 		xmin= 0,
  		xmax= 14000,
        ymax= 0.1,
        xtick= \empty,
        xshift = -1mm,
    	grid style = dashed,
        legend pos = north east,
		legend cell align = {left},
        legend image post style = {scale=0.7},
        legend style={font=\scriptsize, inner xsep=0.2pt, inner ysep=0.1pt},
]
\draw[<->, densely dotted, line width= 0.05cm, Fcolor] (axis cs:400,0.003)--(axis cs:13750,0.003) node [black, pos=0.55, above] {\footnotesize \raisebox{-0.2ex}[0pt][-0.2ex]{$\Wbf\!=\!\vbf\vbf^{\ast}$}};
\draw[<->, densely dotted, line width= 0.05cm, Fcolor] (axis cs:1250,0.0056)--(axis cs:13750,0.0056) node [black, pos=0.5, above] {\footnotesize \raisebox{-0.2ex}[0pt][-0.2ex]{$\Wbf\!=\!\vbf\vbf^{\ast}$}};
\draw[<->, densely dotted, line width= 0.05cm, Fcolor] (axis cs:6400,0.01)--(axis cs:13750,0.01) node [black, pos=0.5, above] {\footnotesize \raisebox{-0.2ex}[0pt][-0.2ex]{$\Wbf\!=\!\vbf\vbf^{\ast}$}};
\addplot[color=red,line width=0.5mm] table[x=X,y=YSDP,col sep=comma] {5pjm.csv};
\addplot[dashed,color=blue,line width=0.5mm] table[x=X,y=YSOCP,col sep=comma] {5pjm.csv};
\addplot[dash dot,color=black,line width=0.5mm] table[x=X,y=YPara,col sep=comma] {5pjm.csv};

\legend{$\mathrm{SDP}$,$\mathrm{SOCP}$,$\mathrm{Parabolic}$}
\end{axis}
\end{tikzpicture}
%%%%%%%%%%%%%%%%%%%%%%%%%%%%%%%%%%%%%%%%%%%%%%%%%%%%%%%%%%%%%%%%%%%%%%%%%
% \vspace{0.1cm}
%%%%%%%%%%%%%%%%%%%%%%%%%%%%%%%%%%%%%%%%%%%%%%%%%%%%%%%%%%%%%%%%%%%%%%%%
\begin{tikzpicture}[slave,scale=0.85]
\centering
\begin{axis}[
		width = 0.52\textwidth,
		height = 0.29\textwidth,
        xmode=normal,
		ymode=normal,	
		ylabel shift = -0.1cm,
  		xmin = 0,
 		xmax = 14000,
        xtickmax = 13000,
        ymin = 14700,
        xtick scale label code/.code={\pgfmathparse{int(#1)}$\mu \times 10^{-\pgfmathresult}$},
 		every x tick scale label/.style={at={(xticklabel cs:0.5)}, anchor = north},
        ytick scale label code/.code={\pgfmathparse{int(#1)}$h_{\mathrm{L}}(\obf,\pbf) \times 10^{-\pgfmathresult}$},
        every y tick scale label/.style={at={(yticklabel cs:0.5)}, anchor = south, rotate = 90},  
    	legend pos=south east,
    	grid style=dashed,
		legend cell align={left},
        legend image post style={scale=0.7},
        legend style={font=\scriptsize, inner xsep=0.2pt, inner ysep=0.1pt},
]
\draw[<->, densely dotted, line width= 0.06cm, Fcolor] (axis cs:400,16675)--(axis cs:13750,16675) node [black, pos=0.6, above] {\footnotesize \raisebox{-0.2ex}[0pt][-0.2ex]{$\Wbf\!=\!\vbf\vbf^{\ast}$}};
\draw[<->, densely dotted, line width= 0.06cm, Fcolor] (axis cs:1250,17305)--(axis cs:13750,17305) node [black, pos=0.6, above] {\footnotesize \raisebox{-0.2ex}[0pt][-0.2ex]{$\Wbf\!=\!\vbf\vbf^{\ast}$}};
\draw[<->, densely dotted, line width= 0.06cm, Fcolor] (axis cs:6400,18010)--(axis cs:13750,18010) node [black, pos=0.53, above] {\footnotesize \raisebox{-0.2ex}[0pt][-0.2ex]{$\Wbf\!=\!\vbf\vbf^{\ast}$}};

\addplot[color=red,line width=0.5mm] table[x=X,y=CSDP,col sep=comma] {5pjm.csv};
\addplot[dashed,color=blue,line width=0.5mm] table[x=X,y=CSOCP,col sep=comma] {5pjm.csv};
\addplot[dash dot,color=black,line width=0.5mm] table[x=X,y=CPara,col sep=comma] {5pjm.csv};

\legend{$\mathrm{SDP}$,$\mathrm{SOCP}$,$\mathrm{Parabolic}$}
\end{axis}
\end{tikzpicture}
\vspace{8mm}
\caption{Behavior analysis of the penalized convex relaxations for different choices of $\mu$ (nesta\_case5\_pjm \cite{coffrin2014nesta}). Top: feasibility violation; Bottom: resulting cost values}
\vspace{-0.5mm}
\label{plt:plot_feas_cost}
\end{figure}
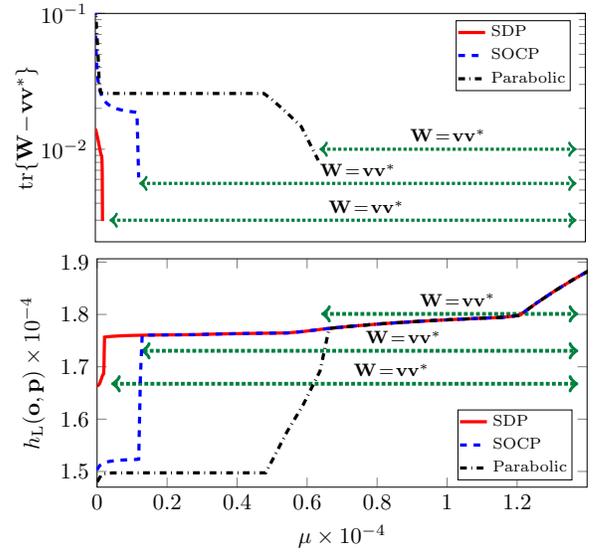
%%%%%%%%%%%%%%%%%%%%%%%?????

To evaluate the sensitivity of the penalized SDP, SOCP, and parabolic relaxations to the choice of penalty parameter $\mu$, we solve the penalized convex relaxation of OPF for the benchmark system nesta\_case5\_pjm from \cite{coffrin2014nesta}, for different values of $\mu$. The results are shown in Figure \ref{plt:plot_feas_cost}.
For this benchmark case, the best-known feasible cost is equal to $17551.89$ \cite{coffrin2014nesta}. The minimum values of $\mu$ that offers tight penalization and its resulting percentage gap with the best-known cost value are, respectively, equal to $213.60$ and $0.08\%$ for SDP relaxation, $1288.88$ and $0.29\%$ SOCP relaxation, and $6628.91$ and $1.02\%$ for parabolic relaxation.  For this experiment, the parameters $\alpha$ and $\eta$ in the equations \eqref{eqM} and \eqref{eqMij} are set to $5$ and $0$, respectively.
According to Figure \ref{plt:plot_feas_cost}, all of the proposed penalized convex relaxations result in near-globally optimal points for a wide range of $\mu$ values. 
%%%%%%%%%%%%%%%%%%%%%%%%%%%%%%%%%%%%%%%%%%%%%%%%%%%%%%
% The minimum values of $\mu$ that offer feasibility for the SDP, SOCP, and Parabolic relaxations, their corresponding objective values
% The lower bound obtained by the SDP relaxation is equal to 16635.78.
% %%
% The minimum values of $\mu$
% SDP: 213.60
% SOCP: 1288.88
% Para:6628.91
% %
% the corresponding objective values
% SDP:17566.65
% SOCP:17602.79
% Para:17732.66
% as well as gaps from global optimality. 
% gap SDP: 0.08
% gap SOCP: 0.29
% gap Para: 1.02
%, which implies the great flexibility
%. of the proposed methods 
%in choosing this parameter. 
As shown by the figure at the bottom, a smaller choice of $\mu$ leads to a lower objective values. The smallest value of $\mu$, which produces a feasible solution for OPF is greater for parabolic relaxation compared to that of SDP and SOCP relaxations.
%%%%%%%%%%%%%%%%%%%%%%%%%%%%%%%%%%%%%%%%%%%%%%%%%%%%%%%%%%%
\begin{table*}
\vspace{0.2cm}
\caption{Result summary for several benchmark systems.}
\scriptsize
	\centering
    \scalebox{1}{
	\begin{tabular}{|@{\,}c@{\,}
    |@{\,}|
    @{\!\!\:\,}c@{\,}| @{\,}c@{\,}| @{\,}c@{\,}| @{\,}c@{\,}| @{\,}c@{\,}| @{\,}c@{\,}| @{\,}c@{\,}| @{\,}c@{\,}
    |@{\,}|
    @{\,}c@{\,}| @{\,}c@{\,}| @{\,}c@{\,}| @{\,}c@{\,}| @{\,}c@{\,}| @{\,}c@{\,}| @{\,}c@{\,}| @{\,}c@{\,}
    | @{\,}|
    @{\,}c@{\,}| @{\,}c@{\,}| @{\,}c@{\,}| @{\,}c@{\,}| @{\,}c@{\,}| @{\,}c@{\,}| @{\,}c@{\,}| @{\,}c@{\,}
    |@{\,}|
    @{\!\!\;\,}c@{\,}|}
	\hline
	\multirow{2}{*}{Test Cases} & \multicolumn{8}{c|@{\,}|@{\,}}{SDP} & \multicolumn{8}{c|@{\,}|@{\,}}{SOCP}& \multicolumn{8}{c|@{\,}|@{\!\!\;\,} }{Parabolic}& \multirow{2}{*}{$c_{s}$}\\
		\cline{2-25}
 		& $\sm{\mu}{6}$ & $\sm{\alpha}{6}$  & $\sm{k_{f}}{6}$ & \sm{\text{GFB\%}}{6} & \sm{\text{GFS\%}}{6} & $\sm{k_{p}}{6}$ & \sm{\text{GPB\%}}{6} & \sm{\text{GPS\%}}{6} & $\sm{\mu}{6}$ & $\sm{\alpha}{6}$ & $\sm{k_{f}}{6}$ & \sm{\text{GFB\%}}{6} & \sm{\text{GFS\%}}{6} & $\sm{k_{p}}{6}$ & \sm{\text{GPB\%}}{6} & \sm{\text{GPS\%}}{6} & $\sm{\mu}{6}$ & $\sm{\alpha}{6}$ & $\sm{k_{f}}{6}$ & \sm{\text{GFB\%}}{6} & \sm{\text{GFS\%}}{6} & $\sm{k_{p}}{6}$ & \sm{\text{GPB\%}}{6} & \sm{\text{GPS\%}}{6} & \\
		\hline
		\hline
        118 & 1e1 & 1 & 1 & 0.00 & 0.01 & 1 & 0.00 & 0.01    & 1e1 & 5 & 1 & 0.05 & 0.06 & 2 & 0.01 & 0.01     & 1e3 & 5 & 2 & 1.38 & 1.39 & 20 & 0.18 & 0.19 & 129654.63 \\ \hline
        300 & 1e3 & 1 & 1 & 0.60 & 0.60 & 4 & 0.03 & 0.04    & 1e2 & 1 & 1 & 0.03 & 0.03 & 2 & 0.01 & 0.01     & 1e3 & 10 & 7 & 0.18 & 0.18 & 12 & 0.08 & 0.08 & 719711.70 \\ \hline
%%%%%%%%%%%%%%%%%%%%%%%%%%%%%%%%% pegase %%%%%%%%%%%%%%%%%%%%%%%%%%%%%%%%%%%%%%%%%%%%
		89pegase  & 1e2 & 1 & 1 & 0.11 & 0.11 & 1 & 0.11 & 0.11     & 1e2 & 1 & 1 & 0.11 & 0.11 & 1 & 0.11 & 0.11      & 1e2 & 10 & 19 & 0.23 & 0.23 & 19 & 0.23 & 0.23 & 5819.67 \\ \hline
		1354pegase & 1e2 & 1 & 1 & -- & 0.16 & 17 & -- & 0.08    & 1e2 & 1 & 1 & -- & 0.16 & 17 & -- & 0.08     & 1e3 & 5 & 14 & -- & 0.34 & 14 & -- & 0.34 & 74062.53 \\ \hline
        \hline   
% %%%%%%%%%%%%%%%%%%%%%%%%%%%%%%%%%%%% BUKHSH
        9mod & 1e4 & 1 & 5 & 0.89 & 11.63 & 15 & 0.03 & 10.87    & 1e4 & 1 & 5 & 0.89 & 11.63 & 15 & 0.03 & 10.87      & 1e4 & 1 & 9 & 0.48 & 11.27 & 17 & 0.03 & 10.87 & 2753.04 \\ \hline

		39mod1  & 1e4 & 5 & 14 & 2.69  & 6.28 & 48 & 0.18 & 3.87     & 1e4 & 1 & 2 & 6.32 & 9.78 & 33 & 0.11 & 3.80     & 1e4 & 5 & 3 & 9.23 & 12.58 & 48 & 0.18 & 3.87 & 10804.08 \\ \hline
        39mod2  & 1e1 & 1 & 3 & 0.04 & 0.19 & 3 & 0.04 & 0.19    & 1e1 & 1 & 3 & 0.04 & 0.19 & 3 & 0.04 & 0.19      & 1e1 & 5 & 3 & 0.07 & 0.22 & 4 & 0.06 & 0.21 & 940.34\\ \hline
        39mod3  & 1e1 & 1 & 2 & -0.01 & 0.29 & 2 & -0.01 & 0.29   & 1e1 & 1 & 2 & -0.01 & 0.29 & 2 & -0.01 & 0.29    & 1e1 & 5 & 4 & -0.01 & 0.29 & 4 & -0.01 & 0.29 & 1884.38 \\ \hline
        39mod4  & 1e1 & 1 & 1 & 2.15 & 2.16 & 3 & 0.03 & 0.04     & 1e1 & 1 & 1 & 2.15 & 2.16 & 3 & 0.03 & 0.04     & 1e1 & 5 & 5 & 1.31 & 1.32 & 11 & 0.06 & 0.08 & 557.08 \\ \hline
        118mod & 1e1 & 1 & 1 & 0.00 & 0.00 & 1 & 0.00 & 0.00     & 1e1 & 1 & 1 & 0.00 & 0.00 & 1 & 0.00 & 0.00      & 1e3 & 5 & 2 & 1.42 & 1.42 & 21 & 0.18 & 0.18 & 129624.98 \\ \hline
        300mod  & 1e4 & 1 & 6 & 0.89 & 1.03 & 24 & 0.27 & 0.41     & 1e4 & 1 & 6 & 0.89 & 1.03 & 24 & 0.27 & 0.41    & 1e4 & 5 & 12 & 1.12 & 1.26 & 39 & 0.51 & 0.64 & 378022.80 \\ \hline
        300mod1 & 1e1 & 1 & 1 & -1.37 & 0.00 & 1 & -1.37 & 0.00  & 1e2 & 1 & 2 & -1.35 & 0.02 & 3 & -1.36 & 0.01     & 1e3 & 5 & 6 & -1.21 & 0.16 & 9 & -1.26 & 0.11 & 474625.99\\ \hline
      \hline
%%%%%%%%%%%%%%%%%%%%%%%%%%%%%%%%%%%%%%%%%%%%%%%%%%%%%%%%%%%%%%%%%%%%%%%%%%%%%%%%%%%%%%
 nesta\_30\_as & 1e1 & 1 & 1 & 0.39 & 0.39 & 2 & 0.01 & 0.01      & 1e1 & 1 & 1 & 0.39 & 0.39 & 2 & 0.01 & 0.01     & 1e2 & 1 & 16 & 0.01 & 0.01 & 16 & 0.01 & 0.01  & 803.13\\ \hline
        nesta\_30\_fsr & 1e1 & 1 & 1 & 0.08 & 0.08 & 2 & 0.01 & 0.01       & 1e1 & 1 & 1 & 0.08 & 0.08 & 2 & 0.01 & 0.01      & 1e2 & 5 & 2 & 2.03 & 2.03 & 18 & 0.07 & 0.07  & 575.77 \\ \hline
        nesta\_30\_ieee & 1e1 & 1 & 1 & 0.27 & 0.27 & 3 & 0.00 & 0.00       & 1e1 & 1 & 1 & 0.27 & 0.27 & 3 & 0.00 & 0.00      & 1e2 & 5 & 9 & 4.52 & 4.52 & 22 & 0.05 & 0.05  & 204.97\\ \hline
        nesta\_39\_epri & 1e2 & 1 & 1 & 0.01 & 0.02 & 1 & 0.01 & 0.02     & 1e2 & 1 & 2 & 0.00 & 0.02 & 2 & 0.00 & 0.02      & 1e3 & 5 & 5 & 0.09 & 0.11 & 8 & 0.04 & 0.06 & 96491.10\\ \hline
        nesta\_57\_ieee & 1e1 & 1 & 1 & 0.10 & 0.10 & 3 & 0.00 & 0.00      & 1e1 & 1 & 1 & 0.10 & 0.10 & 3 & 0.00 & 0.00     & 1e2 & 1 & 9 & 0.05 & 0.05 & 11 & 0.03 & 0.03 & 1143.27\\ \hline
        nesta\_73\_ieee\_rts & 1e1 & 1  & 1 & 0.00 & 0.00 & 1 & 0.00 & 0.00      & 1e1 & 1 & 1 & 0.00 & 0.00 & 1 & 0.00 & 0.00     & 1e3 & 5 & 3 & 0.10 & 0.10 & 6 & 0.01 & 0.01  & 189764.08\\ \hline
		\hline
%%%%%%%%%%%%%%%%%%%%%%%%%%%%%%%%%%%% Nesta_api %%%%%%%%%%%%%%%%%%%%%%%%%%%%%%%%%%%%%%%
        nesta\_30\_as\_api & 1e1 & 1 & 1 & 0.38 & 0.38 & 2 & 0.00 & 0.00      & 1e2 & 1  & 2 & 0.80 & 0.80 & 5 & 0.01 & 0.01     & 1e2 & 10 & 17 & 0.35 & 0.35 & 33 & 0.01 & 0.01  & 570.08\\ \hline
        nesta\_30\_fsr\_api & 1e2 & 5 & 1 & 1.02 & 11.45 & 7 & 0.03 & 10.56      & 1e3 & 5 & 1 & 5.61 & 15.56 & 22 & 0.13 & 10.65     & 1e4 & 5 & 9 & 4.50 & 14.56 & 76 & 0.55 & 11.03 & 327.95\\ \hline
        nesta\_30\_ieee\_api  & 1e1 & 1 & 1 & 0.00 & 0.00 & 1 & 0.00 & 0.00     & 1e1 & 5 & 5 & 0.00 & 0.00 & 5 & 0.00 & 0.00     & 1e2 & 5 & 33 & 0.00 & 0.00 & 33 & 0.00 & 0.00   & 414.99\\ \hline
        nesta\_39\_epri\_api & 1e1 & 1 & 1 & 0.01 & 0.01 & 1 & 0.01 & 0.01      & 1e1 & 1 & 3 & 0.00 & 0.00 & 3 & 0.00 & 0.00     & 1e3 & 5 & 23 & 0.06 & 0.06 & 25 & 0.00 & 0.00 & 7460.37\\ \hline
        nesta\_57\_ieee\_api & 1e1 & 1  & 1 & 0.36 & 0.44 & 2 & 0.00 & 0.09    & 1e1 & 1  & 2 & 0.00 & 0.09 & 2 & 0.00 & 0.09    & 1e2 & 1  & 5 & 0.85 & 0.93 & 11 & 0.02 & 0.11 & 1429.51 \\ \hline
        nesta\_73\_ieee\_rts\_api & 1e3 & 1 & 3 & 1.96 & 6.18 & 25 & 0.21 & 4.50     & 1e3 & 5 & 3 & 2.74 & 6.92 & 50 & 0.60 & 4.87    & 1e4 & 5 & 7 & 5.51 & 9.57 & 79 & 1.90 & 6.11  & 19135.79 \\ \hline
        \hline
%%%%%%%%%%%%%%%%%%%%%%%%%%%%%%%%%%%%%%%%% Nesta_sad %%%%%%%%%%%%%%%%%%%%%%%%%%
        nesta\_30\_as\_sad & 1e1 & 5 & 3 & 0.00 & 0.24 & 3 & 0.00 & 0.24       & 1e2 & 1 & 1 & 1.68 & 1.91 & 3 & 0.00 & 0.24     & 1e2 & 10 & 4 & 2.65 & 2.88 & 15 & 0.00 & 0.24 & 895.34 \\ \hline
        nesta\_30\_fsr\_sad & 1e1 & 1 & 1 & 0.02 & 0.04 & 2 & 0.00 & 0.02       & 1e1 & 1 & 1 & 0.02 & 0.04 & 2 & 0.00 & 0.02    & 1e2 & 5 & 7 & 0.21 & 0.23 & 12 & 0.02 & 0.04 & 576.68 \\ \hline
        nesta\_30\_ieee\_sad & 1e1 & 1 & 1 & 0.27 & 0.27 & 3 & 0.00 & 0.00      & 1e1 & 1 & 1 & 0.27 & 0.27 & 3 & 0.00 & 0.00     & 1e2 & 5 & 9 & 4.52 & 4.52 & 22 & 0.05 & 0.05 & 204.97 \\ \hline
        nesta\_39\_epri\_sad & 1e2 & 1 & 1 & 0.00 & 0.06 & 1 & 0.00 & 0.06     & 1e2 & 1 & 1 & 0.00 & 0.06 & 1 & 0.00 & 0.06     & 1e3 & 5 & 5 & 0.05 & 0.10 & 6 & 0.00 & 0.06  & 96692.45 \\ \hline
        nesta\_57\_ieee\_sad & 1e1 & 1 & 1 & 0.10 & 0.10 & 3 & 0.00 & 0.00     & 1e1 & 1 & 1 & 0.10 & 0.10 & 3 & 0.00 & 0.00    & 1e2 & 1 & 9 & 0.05 & 0.05 & 11 & 0.03 & 0.03   &  1143.27\\ \hline
        nesta\_73\_ieee\_rts\_sad & 1e3 & 1 & 1 & 0.01 & 2.76 & 1 & 0.01 & 2.76      & 1e3 & 5 & 1 & 0.02 & 2.77 & 2 & 0.01 & 2.76   & 1e4 & 5 & 2 & 4.78 & 7.40 & 5 & 0.02 & 2.77 &  221480.59\\ \hline
			\end{tabular}
            }
	\label{tab:Small_tb}
    \vspace{-2.5mm}
\end{table*}

\begin{table}
 \caption{Performance of the SOCP sequential algorithm on a large-scale 13659-bus European system presented in \cite{zimmerman2011matpower}.}
\scriptsize
	\centering
    \scalebox{0.91}{
	\begin{tabular}{|@{\,}c@{\,}
    |@{\,}|
    @{\!\!\:\,}c@{\,}| @{\,}c@{\,}| @{\,}c@{\,}| @{\,}c@{\,}| @{\,}c@{\,}| @{\,}c@{\,}| @{\,}c@{\,}| @{\,}c@{\,}
    | @{\,}|
    @{\!\!\;\,}c@{\,}|}
		\hline
		\multirow{2}{*}{Test Cases}& \multicolumn{8}{c|@{\,}|@{\!\!\;\,}}{SOCP}& \multirow{2}{*}{$c_{b}$}\\
 		\cline{2-9}
 		& $\sm{\mu}{6}$ & $\sm{\alpha}{6}$ & $\sm{k_{f}}{6}$ & \sm{\text{GFB\%}}{6} & $\sm{k_{p}}{6}$ & \sm{\text{GPB\%}}{6} & $\sm{c_{f}}{6}$ & $\sm{c_{p}}{6}$ &\\
		\hline
		\hline
		pegase 13659 & 1e2 & 5 & 12 & 0.18 & 20 & 0.15 & 386805.70 & 386691.22 & 386115.18 \\ \hline
		nesta\_pegase\_13659 & 1e2 & 5 & 12 & 0.18 & 20 & 0.15 & 386807.08 & 386692.66 & 386125.24 \\ \hline
        nesta\_pegase\_13659\_api & 1e2 & 1 & 7 & 1.16 & 20 & 0.32 & 306457.63 & 303854.93 & 302891.76 \\ \hline
        nesta\_pegase\_13659\_sad & 1e2 & 1 & 4 & 0.13 & 20 & 0.05 & 386630.16 & 386352.72 & 386145.99\\
        \hline
	\end{tabular}}
	\label{tab:Large_tb}
    \vspace{-2.5mm}
\end{table}

%%%%%%%%%%%%%%%%%%%%%%%%%%%%%%%%%%%%%%%%%%%%%%%%%%%%%%%%%%%

As our third experiment, we evaluate the performance of the proposed sequential scheme for solving OPF on several benchmark systems. The numerical results are reported in Tables \ref{tab:Small_tb} and \ref{tab:Large_tb}. For all of the test cases, we have initialized Algorithm \ref{alg:1} with flat start, i.e., $\vbf_{\sm{0}{5}}\!=\!\onebf$, $\pbf_{\sm{0}{5}}=\!\pbf_{\mathrm{min}}$, $\qbf_{\sm{0}{5}}\!=\mathbf{0}$, $\vec{\sbf}_{\sm{0}{5}} = \diagrm\{ \vec{\Cbf}\vbf_{\sm{0}{5}}\vbf_{\sm{0}{5}}^{\top}\vec{\Ybf}^{\ast}\}$, and $\cev{\sbf}_{\sm{0}{5}} = \diagrm\{ \cev{\Cbf}\vbf_{\sm{0}{5}}\vbf_{\sm{0}{5}}^{\top}\cev{\Ybf}^{\ast}\}$. Additionally, the parameter $\eta$ is set to zero for all cases.
%
%The \textit{optimality gap} of OPF feasible points are evaluated using the formula, $\mathrm{Gap\%}=100\times\frac{c_{\mathrm{ub}} - c_{\mathrm{lb}}}{c_{\mathrm{ub}}}$, where $c_{\mathrm{ub}}$ denotes the total generation cost of the feasible point from Algorithm \ref{alg:1} %$(\vbf,\pbf+\irm\qbf,\vec\sbf,\cev\sbf)$ 
%and $c_{\mathrm{lb}}$ denotes the optimal cost of unpenalized SDP relaxation without penalty or an upper bound on the globally optimal cost???????????????. The optimality gap shows the maximum possible gap between the exact solution and the unknown globally optimal solution.
%
% we assess the proposed method by using $\mu\in\{1e^i,2e^i,5e^i\}_{i=-1}^{4}$. Note that if the relaxation is exact for a certain value of $\mu_0$, it is guaranteed to be exact for any $\mu>\mu_0$.
To assess the quality of the proposed sequential penalized convex relaxations the following numbers are reported
\begin{itemize}
\item{$c_{b}$: is the best known cost value that is reported by benchmark producers.}
\item{$c_{s}$: is the resulting cost lower bound from unpenalized SDP relaxation.}
\item{$k_{f}$: is the number of first round that produces an OPF feasible point, satisfying $\tr\{\Wbf-\vbf\vbf^{\ast}\}<10^{\smallMinus 7}$.} %the penalized relaxation become exact. Here, the permissible level of inexactness is set to $10^{\smallMinus 7}$ for all test cases. }
\item{$c_{f}$: is the total generation cost associated with the
operating point $(\vbf,\pbf\!+\!\irm\qbf, \vec\sbf,\cev\sbf)$ at round $k_{f}$.}
\item{$\mathrm{GFB\%}=100\times (c_{f} - c_{b})/c_{f}$: is the percentage optimality gap between the cost value at round $k_{f}$ and the best cost value reported by benchmark producers.}
\item{$\mathrm{GFS\%}=100\times (c_{f} - c_{s})/c_{f}$: is the percentage optimality gap between the cost value at round $k_{f}$ and the lower bound from unpenalized SDP relaxation.}
\item{$k_{p}$: In Table \ref{tab:Small_tb}, $k_{p}$ is the first round number whose cost value (without penalty) is not more than 0.01\% improved compared to the previous round. However, for Table \ref{tab:Large_tb}, Algorithm \ref{alg:1} is terminated after 20 rounds and $k_{p}=20$ regardless of the progress.  }
\item{$c_{p}$: is the total generation cost associated with the operating point $(\vbf,\pbf\!+\!\irm\qbf,\vec\sbf,\cev\sbf)$ at round $k_{p}$.}
\item{$\mathrm{GPB\%}=100\times (c_{p} - c_{b})/c_{p}$: is the percentage optimality gap between the cost value at round $k_{p}$ and the best cost value reported by benchmark producers.}
\item{$\mathrm{GPS\%}=100\times (c_{p} - c_{s})/c_{p}$: is the percentage optimality gap between the cost value at round $k_{p}$ and the lower bound from unpenalized SDP relaxation.}
%
%\item{$t(s)$: is the average computation time per round (in seconds).}
\end{itemize}
Note that the GFB\% and GPB\% values reported in Table \ref{tab:Small_tb} are calculated according to the best upper bounds provided by \cite{coffrin2014nesta,bukhsh2013local,zimmerman2011matpower}. %These values are reported for the comparison purposes. 

For all of the test cases reported in Table \ref{tab:Small_tb}, Algorithm \ref{alg:1} equipped with any of the SDP, SOCP, and parabolic relaxations yields fully feasible points within the first few rounds. %Although penalizing the apparent power loss has been previously proposed, there is no other 
%This is the penalization method that is able to produce feasible points to challenging test cases such as case9mod and case39mod1 \cite{molzahn2015solution}. 
As shown in Table \ref{tab:Small_tb}, Algorithm \ref{alg:1} produces feasible points within $0.2\%$ gap from the best reported solutions for benchmark systems case9mod and case39mod1, which is an improvement upon the existing penalization methods \cite{molzahn2015solution}.
%these test cases with the property that the best known upper bounds are at most 0.2\% away from the obtained solutions \cite{zimmerman2011matpower}. 

To verify the scalability of the proposed method, we conduct experiments on the largest available benchmark instances from \cite{zimmerman2011matpower} and \cite{coffrin2014nesta}. The results are reported in Table \ref{tab:Large_tb}. Algorithm \ref{alg:1} equipped with SOCP relaxation finds fully feasible solutions that are not more than 0.4\% away from the upper bounds obtained by solving OPF using \matpower.

%=================== Conclusion ==========================================
\section{Conclusion}
This paper is concerned with the AC optimal power flow (OPF) problem. We first consider two common practice semidefinite programming (SDP) and second order cone programming (SOCP) relaxations of OPF. Due to the computational complexity of conic optimization, we propose an efficient alternative, called parabolic relaxation, which transforms arbitrary non-convex quadratically constrained quadratic programs (QCQPs) to convex QCQPs. %Since the solutions obtained by these relaxations are not necessarily feasible for the OPF problem, 
Additionally, we propose a novel penalization method which is guaranteed to provide feasible points for the original non-convex OPF, under certain assumptions. 
By applying the proposed penalized convex relaxations sequentially, we obtained fully feasible points with promising global optimality gaps for several challenging benchmark instances of OPF.
{
% \small
\bibliographystyle{IEEEtran}
\bibliography{IEEEabrv,egbib}
}

\end{document}